\documentclass[11pt]{article}

\usepackage{graphicx}
\usepackage{mathrsfs}
\usepackage{multirow}
\usepackage{amssymb, amsmath, amsthm}
\usepackage[ruled,vlined,linesnumbered]{algorithm2e}
\usepackage[letterpaper, total={6.5in, 9in}]{geometry}
\usepackage{color}
\usepackage[table,xcdraw]{xcolor}

\usepackage[font=small,labelfont=bf]{caption}
\usepackage{float}
\usepackage{bm}  
\usepackage{xfrac}

\usepackage[sort&compress,numbers]{natbib}
\usepackage[colorlinks=true]{hyperref}
\usepackage{nameref}
\makeatletter
\newcommand{\nametag}[1]{%
	\def\@currentlabelname{#1}%
}
\makeatother
\usepackage{needspace}
\usepackage{textcomp}
\usepackage{tikz,pgfplots}

\usepackage{authblk}
\usepackage{booktabs}
\usepackage{setspace}

\usepackage[capitalize]{cleveref}
\crefformat{equation}{(#2#1#3)}
\crefrangeformat{equation}{(#3#1#4)--(#5#2#6)}

\makeatletter
\newtheorem*{rep@theorem}{\rep@title}
\newcommand{\newreptheorem}[2]{%
	\newenvironment{rep#1}[1]{%
		\def\rep@title{#2 \ref{##1}}%
		\begin{rep@theorem}}%
		{\end{rep@theorem}}}
\makeatother

\newreptheorem{theorem}{Theorem}
\newreptheorem{lemma}{Lemma}
\newreptheorem{proposition}{Proposition}
\newtheorem{theorem}{Theorem}[section]

\newtheorem{lemma}[theorem]{Lemma}
\newtheorem*{lemma*}{Lemma}
\newtheorem{proposition}[theorem]{Proposition}
\newtheorem*{proposition*}{Proposition}
\newtheorem{corollary}[theorem]{Corollary}

\newtheorem*{observation*}{Observation}

\newtheorem*{question*}{Question}

\newtheorem*{definition*}{Definition}
\newtheorem*{problem*}{Problem}
\newtheorem*{InformalDefinition*}{Definition (Informal)}

\newtheorem*{notation*}{Notation}

\theoremstyle{definition}

\newtheorem*{remark*}{Remark}

\newtheorem*{example*}{Example}

\newcommand{\ip}[2]{\left\langle #1 , #2 \right\rangle}    

\def\beq{\begin{equation}}
\def\eeq{\end{equation}}

\def\fnote#1{\footnote}

\DeclareMathOperator{\solve}{\mathsf{APPROXMIN}}
\DeclareMathOperator{\KrylovSym}{\mathsf{Krylov-Sym}}

\newcommand{\boldL}{\boldsymbol{\ell}}

\newcommand{\KrylovSymShort}{\hyperref[alg:Krylov-iteration-sym]{\textsf{Krylov-Sym}}}

\newcommand{\fdb}{Fa\`a di Bruno\xspace}

\def\cC{{\cal C}}

\def\cM{{\cal M}}
\def\cN{{\cal N}}
\def\cO{{\cal O}}
\def\cP{{\cal P}}
\def\cQ{{\cal Q}}

\def\cT{{\cal T}}

\def\cV{{\cal V}}

\def\cX{{\cal X}}

\def\sfB{{\mathsf{B}}}
\def\sfC{{\mathsf{C}}}

\newcommand{\bbN}{\mathbb{N}}

\newcommand{\bbR}{\mathbb{R}}

\DeclareMathOperator{\Skew}{Skew}
\DeclareMathOperator{\grad}{grad}
\DeclareMathOperator{\hess}{Hess}

\DeclareMathOperator{\ordpart}{\mathcal{OP}}

\DeclareMathOperator{\vl}{vl} 
\DeclareMathOperator{\partrans}{PT}
\DeclareMathOperator{\VT}{VT}

\DeclareMathOperator{\curvTens}{\mathcal{R}}
\DeclareMathOperator{\radius}{rad}

\DeclareMathOperator{\Id}{Id}

\DeclareMathOperator{\Diag}{Diag}

\DeclareMathOperator{\Exp}{Exp}
\DeclareMathOperator{\St}{St}

\DeclareMathOperator{\Sym}{Sym}

\usepackage{scalerel}

\def\dim{\mathop{{\rm dim}\,}}

\DeclareMathOperator{\Ker}{Ker}

\def\log{\mathop{{\rm log}}}

\DeclareMathOperator{\Gr}{Gr}
\newcommand{\regparam}{\alpha}
\DeclareMathOperator{\insertind}{Ins}
\DeclareMathOperator{\appendind}{App}

\makeatletter
\newcommand{\namedpart}[2]{%
	\item \begingroup
	\def\@currentlabelname{#1}
	\label{#2}%
	\endgroup
	\textit{(#1)}\ %
}
\makeatother

\makeatletter
\newenvironment{multicases}[1]
  {\let\@ifnextchar\new@ifnextchar
   \left\lbrace\def\arraystretch{1.2}%
   \array{@{}l*{#1}{@{\quad}l}@{}}}
  {\endarray\right.\kern-\nulldelimiterspace}
\makeatother

\title{An Invitation to Higher-Order Riemannian Optimization:\\ Optimal and Implementable Methods
\thanks{\textbf{Funding}: 
	The authors are supported by the NSF grant \#2410328.
}
}
\date{January 29, 2026}

\SetKwFunction{FMain}{MainProcedure}
\SetKwFunction{FSub}{SubProcedure}
\SetKwProg{Fn}{Function}{:}{\KwRet}

\allowdisplaybreaks[4]

\begin{document}
\author{%
	David Huckleberry Gutman\thanks{Department of Industrial and Systems Engineering, Texas A\&M University,
		College Station, TX 77843. \href{mailto:dhgutman@tamu.edu}{dhgutman@tamu.edu}}
	\qquad
	George Lobo\thanks{Department of Industrial and Systems Engineering, Texas A\&M University,
		College Station, TX 77843. \href{mailto:george.lobo@tamu.edu}{george.lobo@tamu.edu}}
}

\maketitle

\begin{abstract}
This paper presents the first optimal-rate $p$-th order methods with $p\geq 1$ for finding first and second-order stationary points of non-convex smooth objective functions over manifolds.  In contrast to the geodesically convex setting, we establish that the optimal oracle complexity of non-convex optimization over manifolds matches that over Euclidean space. Along with this complexity analysis, we introduce a framework for studying arbitrary-order regularity of pullbacks of the objective function to tangent spaces via a retraction. To the best of our knowledge, this is the first use in optimization of pullback connections and the Sasaki metric to analyze these pullbacks. This approach yields simple derivative bounds based on a new covariant \fdb formula. For $p=3$, our methods are fully implementable via a new Krylov-based framework for minimizing quartically regularized cubic polynomials. This is the first Krylov method for this class of polynomials and may be of independent interest beyond Riemannian optimization.
\end{abstract}

\section{Introduction}\label{sec:introduction}
\enlargethispage{\baselineskip}

We consider the problem of finding first- and second-order stationary points of a smooth, nonconvex objective function over a Riemannian manifold,
\[
\min_{x \in \mathcal M} f(x),
\]
commonly referred to as Riemannian optimization (RO). We assume oracle access to derivatives of $f$ up to order $p\geq 1$. In Euclidean space, higher-order methods for this problem are well understood: adaptive regularization achieves optimal worst-case iteration complexity for finding stationary points \cite{Birgin17,Carmon20,Cartis20,Cartis18,Jiang21,Gasnikov19,Nesterov21a,Arjevani19}. It is therefore natural to ask:
\begin{quote}
\emph{Do these optimal complexity guarantees extend to manifolds? If so, can the resulting methods be implemented in practice?}
\end{quote}
This paper answers both questions in the affirmative, by establishing a complexity-theoretic and algorithmic foundation for higher-order RO. Our intent is not to demonstrate numerical supremacy, but to clarify what is possible in principle and to invite further research in this direction.

However, the existing literature gives little reason to expect such unequivocally positive answers. In contrast to the Euclidean setting, the complexity landscape of RO is more rugged and unevenly charted even at the level of first- and second-order methods. For geodesically convex optimization, the Riemannian analogue of convex optimization, \cite{Criscitiello21,Hamilton21} show that optimal first-order Euclidean rates fail to carry over to manifolds. This exposes a fundamental gap between Euclidean and Riemannian optimization. These rates can be recovered only under additional assumptions, such as restricting the iterates to a pre-compact subset of the manifold \cite{Kim22,Feng25}.

In the nonconvex regime, the picture is more encouraging but remains limited to orders $p=1,2$. For first-order methods applied to nonconvex and \(L\)-smooth objectives, optimal convergence rates match those in Euclidean space \cite{Boumal23,Nesterov18}. Similarly, for second-order methods under Lipschitz continuity of the Hessian, optimal convergence rates extend to the Riemannian setting \cite{Boumal19,Agarwal21}. To the best of our knowledge, no existing work formulates RO methods of order greater than two, leaving open whether optimal nonconvex Euclidean complexity guarantees persist on manifolds.

The optimal complexity of higher-order methods in Euclidean optimization has long been a topic of interest in the optimization community. The provably stronger worst-case complexity guarantees of higher-order methods, relative to first- and second-order ones, illustrate their potential for practical gains. In the nonconvex and unconstrained setting with $p \geq 2$, the first optimal-rate method for finding first-order stationary points was proposed in \cite{Birgin17}, with matching lower complexity bounds subsequently established in \cite{Carmon20}. These results generalize the adaptive regularization with cubics framework developed for the case $p=2$ in \cite{Cartis11a,Cartis11b}, which has since been extended to accommodate constrained optimization problems \cite{Cartis20,Cartis18}. Parallel lines of work have established optimal complexity guarantees for higher-order methods in other settings, including unconstrained smooth and composite convex optimization \cite{Jiang21,Gasnikov19}, together with corresponding lower bounds \cite{Nesterov21a,Arjevani19}. Taken together, these results illustrate both the maturity of the theory and the breadth of interest in optimal higher-order complexity, providing a natural benchmark for extensions beyond the Euclidean setting.

\begin{table}[h]
    \centering
    \setlength{\arrayrulewidth}{1pt}
    \renewcommand{\arraystretch}{1.5}
    \begin{tabular}{|>{\centering\arraybackslash}m{0.6cm}|
                        >{\centering\arraybackslash}m{1.6cm}|
                        >{\centering\arraybackslash}m{3.0cm}|
                        >{\centering\arraybackslash}m{3.0cm}|}
        \cline{3-4}
        \multicolumn{2}{c|}{} &
        \multicolumn{1}{c|}{\textbf{g-Convex}} &
        \multicolumn{1}{c|}{\textbf{Non-convex}} \\
        \hline
        \multirow{3}{*}{\rotatebox[origin=c]{90}{\textbf{Order}}} &
            $p = 1$   & $\neq^{*}$ & $=$ \\
        \cline{2-4}
        &   $p = 2$   & \multirow{2}{*}{?} & $=$ \\
        \cline{2-2}\cline{4-4}
        &   $p \ge 3$ &                & $= \; (\textbf{This paper})$ \\
        \hline
    \end{tabular}
    \caption{Description of regimes in which optimal rates for RO match their Euclidean counterparts by level of convexity and order of derivatives. ``$=$'', ``$\neq^*$", and ``$?$" respectively denote settings in which the rates equal, unequal (except under special conditions), and unknown.}
    \label{tab:g_convex_orders_box}
\end{table}

Higher-order methods for RO inherit the implementability challenge of Euclidean methods while introducing an additional obstacle specific to manifolds. In particular, higher-order RO faces two impediments: practical subproblem solution schemes and identification of suitable higher-order retractions. The former is inherited from Euclidean higher-order optimization, where each iteration requires minimizing a generally nonconvex higher-order polynomial model in high dimension. In RO, this difficulty is more severe, as scalable methods must operate without pre-computed bases of tangent spaces, effectively forcing Krylov-type approaches. Consequently, higher-order methods in RO must capitalize on existing Euclidean subproblem solvers, while augmenting them to operate in basis-free geometric settings.

In the Euclidean setting, the design of such solvers is already a central issue in the cubic $(p = 3)$ case and drives a substantial literature. When the regularized cubic subproblem is convex, \cite{Nesterov21a,Nesterov21b,Nesterov21c,Nesterov23,Nesterov25} develop second-order methods for its solution. In the general nonconvex setting, complementary approaches have been proposed, including semidefinite relaxations~\cite{Ahmadi24} and secular equation--based approaches for prudently simplified models~\cite{Cartis25_CQR,Cartis25_QQR,Cartis25_DTM}. For $p = 2$, adaptive regularization methods rely on such solver frameworks, as demonstrated by the ARC method \cite{Cartis11a,Cartis11b} and its Riemannian extension \cite{Agarwal21}. Notably,~\cite{Cartis25_CQR,Cartis25_QQR,Cartis25_DTM} observe that Krylov-based solvers for higher-order polynomial models remain undeveloped, despite their potential scalability. The need for such methods is especially acute in RO, where tangent spaces often lack obvious and convenient bases.

\subsection{Contributions Summary \& Paper Outline}

This paper presents the first unified study of higher-order optimal complexity in RO together with the implementability of the resulting algorithms. To this end, the paper is organized around two complementary themes: complexity and implementability.

The theoretical developments described below introduce optimal-rate methods and detail how their smoothness constants depend on the objective function and the retraction.

\begin{itemize}
	\item \underline{\nameref{sec:p-RAR}} (\cref{sec:p-RAR}): This section introduces the $p$-th order Riemannian Adaptive Regularization ($p$-RAR) method, the first optimal-rate $p$-th order algorithm for smooth, non-geodesically convex optimization on manifolds. In Theorems \ref{thm:convergence-analysis:first-order} and \ref{thm:convergence-analysis:second-order}, we show that $p$-RAR extends the optimal worst-case iteration complexity guarantees for finding first- and second-order stationary points from Euclidean space to the Riemannian setting. Our method is flexible enough to account not just for the exponential map, but for any retraction on a manifold. The analysis develops Riemannian generalizations of smoothness assumptions common in the Euclidean literature.

\item \underline{\nameref{sec:retraction-properties}} (\cref{sec:retraction-properties}): This section gives an explicit, higher-order characterization of how the objective function and the chosen retraction jointly determine the smoothness constants $L_1$ and $L_2$ arising in the analysis of $p$-RAR. These constants govern the first- and second-order convergence rates established for the method. In Theorems \ref{thm:smoothness-bound-first-order} and \ref{thm:smoothness-bound}, we derive explicit, order-accurate bounds on $L_1$ and $L_2$. The bounds are expressed using adaptations of derivative semi-norms familiar from the study of function compositions. They scale appropriately with the order of differentiation and clarify how retractions influence higher-order regularity. 

Our analysis combines two underutilized geometric tools in RO, the pullback connection and the Sasaki metric on the double tangent bundle. Together with a previously unavailable covariant Fa{\'a} di Bruno formula, this framework interprets higher-order covariant derivatives of a retraction as genuine tensors. This yields short, transparent proofs and sheds new light on prior regularity analyses for retraction-based pullback models (\cref{cor:p-order-nice-verification}).

\end{itemize}

In the implementability sections, we provide the first concrete progress toward overcoming the principal barriers to implementing higher-order RO methods for $p \ge 3$. These barriers include the unavailability of practical retractions and subproblem solvers that operate without pre-computed bases. We devote individual sections to each barrier, followed by a numerical experiments section.

\begin{itemize}
	\item \underline{\nameref{sec:retractions-practical}} (\cref{sec:retractions-practical}): This section establishes that the projected polynomial-based retractions of \cite{Gawlik18} for the Stiefel and Grassmannian manifolds are suitable for implementation within the $p$-RAR framework. We show that these retractions are well-defined on the entire tangent bundle (\cref{lemma:well-defined-retraction}). We also show that they are $p$-th order for an appropriate choice of projected polynomial degree (\cref{thm:p-th-order-Gawlik}). This result builds on the extensive analysis of \cite{Gawlik18} and reconciles their notion of $p$-th order retractions with the natural generalization for  RO.
	
	\item \underline{\nameref{sec:subproblem-solver}} (\cref{sec:subproblem-solver}): In this section, we propose the first Krylov-based framework for minimizing quartically-regularized cubic polynomials, which arise as the defining subproblems of the $3$-RAR method. The framework does not require pre-computation of a basis for the underlying vector space, a feature that is particularly important in high-dimensional settings. Instead, the polynomial is minimized over a sequence of growing Krylov subspaces spanned by linearly independent directions. This opens the opportunity to quickly identify an approximate stationary point within a low-dimensional subspace, when such structure is present. The framework can incorporate a variety of existing subproblem solvers, including recent empirically effective methods \cite{Cartis25_CQR,Cartis25_QQR,Cartis25_DTM} for the $p=3$ setting. Although the presentation focuses on the $p=3$ case, the framework naturally extends to higher-order subproblems with $p>3$.

	\item \underline{\nameref{sec:numerics}} (\cref{sec:numerics}): In this section, we compare the performance of several subproblem solvers \cite{Cartis25_CQR,Cartis25_QQR,Cartis25_DTM} when integrated into the $3$-RAR framework. We apply $3$-RAR with each solver to Brockett cost minimization over the Stiefel manifold, a standard benchmark problem in RO. Our goal is not to demonstrate numerical superiority over existing RO methods, but to assess which subproblem solver is most compatible with, and effective within, the proposed framework.

\end{itemize}

The remaining two sections serve supporting roles. \cref{sec:setup} delineates the notation used throughout the paper. \cref{sec:conclusion} concludes the paper with an overview of opportunity-rich future research directions in higher-order Riemannian methods, suggested by the perspective developed in this paper.

\section{Notation}\label{sec:setup}

This section introduces the basic notation used throughout the paper. For a thorough description of these constructs, we refer the reader to the approachable and complete text \cite{Boumal23} for RO-specific material, and to \cite{Lee09,Lee12_smooth,Lee18_riemann} for general diffeo-geometric foundations.

Throughout, let $\cM$ be a Riemannian manifold. For $x\in\cM$, we denote the tangent space by $\cT_x\cM$, the Riemannian metric by $\ip{\cdot}{\cdot}_x$, the associated norm by $\|\cdot\|_x$, and the cotangent space by $\cT_x^*\cM$. When the base point or underlying manifold is clear from context, we omit it from subscripts and superscripts on metrics, norms, and connections. In this case, we may write a tangent vector as either $(x,v)$, to reflect the basepoint, or simply $v$.

We write $\cT\cM$ for the tangent bundle of $\cM$ and $\pi^{\cM}$ for the canonical bundle projection, so that $\pi^{\cM}(v)=x$ for $v\in\cT_x\cM$. We denote by $\Gamma(\cT\cM)$ the set of smooth vector fields on $\cM$, and by $\nabla^{\cM}$ the Levi-Civita connection. Unless stated otherwise, all connections appearing in this paper are Levi-Civita, with the exception of the pullback connection introduced in \cref{sec:retraction-properties}.

\subsection{Differentiation on Vector Spaces and Manifolds}

Given smooth manifolds $\cM$ and $\cN$ and a $C^1$ map $F:\cM\to\cN$, we denote by $dF_x:\cT_x\cM\to\cT_{F(x)}\cN$ the differential of $F$ at $x$. When $\cN=\bbR$, the differential $df_x$ is identified with a covector in $\cT_x^*\cM$, and via the Riemannian metric with the gradient $\grad f(x)\in\cT_x\cM$. When $\cM$ is a finite-dimensional vector space, we write $Df(x)$ for the classical derivative.

To describe higher-order derivatives on vector spaces, we introduce the following notation for a $k$-linear map (i.e. tensor) $T:V^k\to W$ between vector spaces $V$ and $W$. For $v_1,\ldots,v_k,v\in V$, we write $T\left(\otimes_{i=1}^k v_i\right) := T(v_1,\ldots,v_k)$ and $T(v^{\otimes k})=T(v,\ldots,v)$. When $V$ is an inner product space and $W=\bbR$, we define the operator norm by $\|T\|_{op} := \sup\left\{T\left(\otimes_{i=1}^k v_i\right) : \|v_i\|=1,\ i=1,\ldots,k \right\}$.

For a $C^p$ function $f:V\to\bbR$ on a finite-dimensional vector space $V$, we write $D^p f(x)$ for its $p$-th derivative at $x$, viewed as a symmetric $p$-linear map on $V$. For $f\in C^p(\cM)$, we write $\nabla^p f(x)$ (or $\nabla^p f_x$) for its $p$-th covariant derivative at $x$, taken with respect to the Levi-Civita connection. For $p\geq 2$, $\nabla^p f(x)$ is a covariant $p$-tensor on $\cT_x\cM$. When the base point is clear from context, we suppress it and simply write $\nabla^p f$. The Riemannian Hessian at $x\in\cM$ is the unique self-adjoint operator $\hess f(x)$ such that $\nabla^2 f(x)[u,v]=\ip{u}{\hess f(x)v}$ for all $u,v\in\cT_x\cM$.

\subsection{Retractions and Vector Transports}

A retraction is a smooth map $R:\cT\cM\to\cM$ such that, for each $x\in\cM$ with zero vector $0_x\in\cT_x\cM$, one has $R(0_x)=x$ and $d(R|_{\cT_x\cM})_{0_x}=\Id_{\cT_x\cM}$ under the canonical identification $\cT_{0_x}\cT_x\cM\simeq\cT_x\cM$. We will impose additional assumptions on retractions in later sections. The Riemannian exponential map $\Exp:\cT\cM\to\cM$ is the canonical example of a retraction, defined locally in general and globally on metrically complete manifolds.

Given a retraction $R$, a vector transport is a smooth map $\VT:\{(u,v)\in\cT\cM^2 : \pi(u)=\pi(v)\}\to\cT\cM$ such that, for each $u\in\cT\cM$, the map $\VT_u:\cT_{\pi(u)}\cM\to\cT_{R(u)}\cM$ is linear and satisfies $\VT_{0_x}=\Id_{\cT_x\cM}$. Standard examples include parallel transport, which is associated with $\Exp$, and the differentiated retraction. Let $T$ be a covariant $\ell$-tensor field on $\cM$, that is, a smooth assignment $x\mapsto T_x\in(\cT_x^*\cM)^{\otimes\ell}$. For $x\in\cM$, we define the pullback of $T$ to $\cT_x\cM$ induced by $\VT$ by $(\VT_w)^*T_x\left(\otimes_{i=1}^\ell v_i\right)
:= T_{R(w)}\left(\otimes_{i=1}^\ell\VT_w(v_i)\right)$, for all $v_1,\ldots,v_\ell,w\in\cT_x\cM$.

\section{Optimal-Rate Theory, Part I: Algorithms}\label{sec:p-RAR}
\enlargethispage{\baselineskip}

In this section, we introduce the $p$-th Order Riemannian Adaptive Regularization method ($p$-RAR), our main algorithmic contribution. This method establishes that the oracle complexity of higher-order nonconvex optimization on Riemannian manifolds matches that in Euclidean space. It simultaneously extends two foundational methods along different axes: Euclidean adaptive regularization (AR$p$) methods \cite{Birgin17} to manifolds and the Riemannian ARC framework \cite{Agarwal21} beyond the second-order setting. We establish optimal first- and second-order convergence rates matching those of Euclidean AR$p$ methods in Theorems~\ref{thm:convergence-analysis:first-order} and~\ref{thm:convergence-analysis:second-order}, presented in Sections~\ref{sec:first-order-guarantees} and~\ref{sec:second-order-guarantees}. These rates hold under smoothness assumptions that extend the standard ones used in higher-order method analysis to manifolds through two mechanisms: Taylor remainder bounds and polynomial majorization. The study of how derivatives of the objective and the retraction jointly determine the associated smoothness constants is deferred to Section~\ref{sec:retraction-properties}.

We now motivate the $p$-RAR method and introduce the notation for the Taylor-based models that underpin it. Like its Euclidean ancestor, the AR$p$ method \cite{Birgin17}, $p$-RAR computes its next iterate by approximately minimizing a regularized $p$-th order Taylor model of the objective. On a manifold, such models are naturally constructed on tangent spaces, which are finite-dimensional vector spaces, by pulling back the objective through a retraction. For a point $x\in\cM$, we write the pullback of the objective to the tangent space $\cT_x\cM$ as $\hat{f}_x := f\circ R|_{\cT_x\cM}$. We denote by $T_p\hat{f}_x(0,v)$ the $p$-th order Taylor expansion of $\hat{f}_x$ at the origin of $\cT_x\cM$, given explicitly by
\[
T_p\hat{f}_x(0,v)
= \sum_{j=0}^p \frac{1}{j!}\, D^j \hat{f}_x(0)\left[v^{\otimes j}\right].
\]
We define $T_{p-1}D\hat{f}_x(0,v)$ and $T_{p-2}D^2\hat{f}_x(0,v)$ analogously as the Taylor expansions of the first and second derivatives of the pullback. The corresponding regularized $p$-th order model is defined by
\[
m_p(\regparam,x;v)
:= T_p\hat{f}_x(0,v) + \frac{\regparam}{p+1}\|v\|_x^{p+1},
\]
and serves as the local approximation minimized by the method. At each iteration, $p$-RAR approximately minimizes $m_p(\regparam_i,x_i;\cdot)$ over $\cT_{x_i}\cM$; the precise conditions defining approximate minimization are specified by the \textit{Approximate Model Minimization} step. 

We can now present the $p$-RAR algorithm. Throughout, we consider two execution modes. In first-order mode, the method enforces only the \textit{Approximate Model Minimization} conditions
\eqref{eq:alg-model-decrease-condition} and terminates once $\|\grad f(x_i)\|_{x_i}\leq \epsilon_1$. In second-order mode, the method additionally enforces the \textit{Second-Order Model Optimality} condition \eqref{eq:alg-second-order-condition} at every iteration and terminates only once both $\|\grad f(x_i)\|_{x_i}\leq \epsilon_1$ and $\lambda_{\min}[\hess f(x_i)]\geq -\epsilon_2$ hold.\vspace{1em}

\begin{algorithm}[H]\caption{$p$-th Order Riemannian Adaptive Regularization ($p$-RAR)}\label{alg:r-RAR}
	\small
	\KwData{Initial point $x_0$, first- and second-order error tolerances $\epsilon_1,\epsilon_2>0$, and parameters $\theta,\eta_1,\eta_2,\gamma_1,\gamma_2,\gamma_3,\regparam_{\min}$, and $\regparam_0>0$ satisfying\vspace{-.5em}
		\[
		\theta>0,\quad \regparam_{\min}\in(0,\regparam_0],\quad 0<\eta_1\leq\eta_2<1,\quad\text{ and }0<\gamma_1<1<\gamma_2<\gamma_3
		\]} 
	\vspace{-.5em}
	Initialize $i=0$\;
	\While{$\|\grad f(x_i)\|_{x_i}>\epsilon_1$ (or $\lambda_{\min}[\hess f(x_i)]<-\epsilon_2$ \emph{[if 2\textsuperscript{nd}-order mode]}) }{
		\textit{Approximate Model Minimization}: Select $v_i\in \cT_{x_i}\cM$ such that both conditions hold:
		\begin{align}
			\begin{split}\label{eq:alg-model-decrease-condition}
				m_p(\regparam_i,x_i;v_i)&\leq m_p(\regparam_i,x_i;0)\\
				\left\|Dm_p(\regparam_i,x_i;v_i)[\cdot]\right\|_{op}&\leq\theta\|v_i\|_{x_i}^p.
			\end{split}
		\end{align}
		
		\textit{Second-Order Model Optimality (if 2\textsuperscript{nd}-order mode)}: Ensure $v_i$ satisfies
		\begin{equation}\label{eq:alg-second-order-condition}
			\lambda_{\min}\left(D^2m_p(\regparam_i,x_i;v_i)[\cdot]\right)\geq-\theta\|v_i\|_{x_i}^{p-1}.
		\end{equation}
		
		\textit{Sufficient Decrease Measurement}: Let
		\begin{equation}\label{eq:rho-computation}
			\rho_i:=\frac{f(x_i)-f(R(v_i))}{f(x_i)-T_p \hat{f}_{x_i}(0,v_i)}\;
		\end{equation}
		
		\textit{Iterate \& Parameter Update}: Update $x_{i+1}$ and  $\regparam_{i+1}$ according to the formulas
		\begin{align}
			x_{i+1}&=\begin{cases}
				R(v_i), & \rho_i\geq\eta_1\\
				x_i, & \text{o.w.}
			\end{cases}\label{eq:alg-x-update}\\
			\regparam_{i+1}&\in 
			\begin{multicases}{2}
				[\max\{\regparam_{\min},\gamma_1\regparam_i\},\regparam_i], & \rho_i\geq\eta_2 & \text{(Very Successful)}\\
				[\regparam_i,\gamma_2\regparam_i], & \rho_i\in[\eta_1,\eta_2) & \text{(Successful)}\\
				[\gamma_2\regparam_i,\gamma_3\regparam_i], & \rho_i<\eta_1 & \text{(Unsuccessful)}
			\end{multicases}\;\label{eq:alg-sigma-update}
		\end{align}
		Set $i=i+1$\;
	}
	\KwResult{$x\in\cM$ such that $\|\grad f(x)\|_x\leq\epsilon_1$ (and $\lambda_{\min}[\hess f(x)]\geq-\epsilon_2$ [if 2\textsuperscript{nd}-order mode])}
\end{algorithm}
\vspace{1em}
To complete the introduction of the $p$-RAR method, we now describe its algorithmic structure at a high level. After computing an approximate minimizer of the regularized model in the \textit{Approximate Model Minimization} step, $p$-RAR evaluates the resulting trial step using the \textit{Sufficient Decrease Measurement} step. In this step, the method compares the actual decrease in the objective with the decrease predicted by the Taylor-based model and forms the ratio $\rho_i$ of objective decrease to model decrease. The value of $\rho_i$ determines whether the iteration is classified as successful or unsuccessful. In the subsequent \textit{Iterate and Parameter Update} step, successful iterations update the current point via the retraction and decrease the regularization parameter $\regparam_i$, while unsuccessful iterations retain the current point and increase $\regparam_i$. This stepwise control logic mirrors that of classical AR$p$ methods and is central to both the global convergence guarantees and the complexity analysis developed in the following subsections.

\subsection{First-Order Convergence Guarantees}\label{sec:first-order-guarantees}

In this subsection, we prove the \nameref{thm:convergence-analysis:first-order} (\cref{thm:convergence-analysis:first-order}), the first of our two main $p$-RAR convergence results. It guarantees that the rate at which $p$-RAR finds first-order stationary points matches the known optimal rate for $p$-order smooth optimization in Euclidean space \cite{Birgin17}. Specifically, it establishes that $p$-RAR obtains a first-order stationary point, i.e.  $x\in\cM$ with $\|\grad f(x)\|_x<\epsilon_1$, in at most
\[
\cO\left[\left(\frac{1}{\epsilon_1}\right)^{\frac{p+1}{p}}\right]
\]
calls to the $k\leq p$ derivative oracles $x\mapsto D^k \hat{f}_x(0)$. Note that \cref{prop:pullback-to-covariant-derivative-formula} will show $\nabla^k f_x= D^k \hat{f}_x(0)$ under suitable conditions on $R$. Thus, this subsection unequivocally validates one of this article's central themes: $p$-th order smooth nonconvex optimization on manifolds is no harder than it is on Euclidean space from a complexity standpoint. 

We structure the proof of the \nameref{thm:convergence-analysis:first-order} in three steps. First, we introduce and justify the regularity conditions on the objective and retraction that underpin higher-order Riemannian analysis. These conditions, called $(L,p,R)$-majorization smoothness and $(L,p,R,\VT)$-smoothness, are deliberately formulated to be checkable, and their verification and interpretation are the focus of the next section. Second, we record routine consequences of the assumptions, including bounds on the regularization parameters and the iteration count, using standard AR$p$ arguments. Third, and most critically, we use our higher-order smoothness assumptions to bound the gradient norm at the next iterate in terms of $v_i$. This is formalized in the \nameref{lemma:model-to-gradient-decrease} (\cref{lemma:model-to-gradient-decrease}) and used to prove the \nameref{thm:convergence-analysis:first-order}.

The inaugural step in our analysis is to describe the regularity conditions on $f$ and $R$ that support our convergence analysis. As with many Riemannian generalizations of Euclidean optimization algorithms, $p$-RAR requires a technically subtle reformulation of classical smoothness assumptions into geometric counterparts. Said reformulations should support optimal convergence rates, apply broadly, and admit verification, goals addressed respectively by the convergence analysis here and \cref{sec:retraction-properties}. Our analysis relies on two inequivalent conditions on $f$ and $R$ that generalize Lipschitz continuity of the $p$-th derivative: polynomial majorization of pullbacks and bounds on Taylor remainders. In Euclidean space these conditions are equivalent to $p$-th order Lipschitz smoothness, but on manifolds they generally diverge, even for $p=1$ with the exponential map as retraction \cite[Chapter 10]{Boumal23}. We emphasize that both conditions are unavoidable: $(L,p,R)$-majorization smoothness governs acceptance and iteration complexity, while $(L,p,R,\VT)$-smoothness relates approximate model minimization to gradient norms across tangent spaces.

The first of these conditions ensures that the models used by $p$-RAR upper bound the pullback objective on each iterate’s tangent space. It generalizes the classical quadratic majorization characterization of $L$-smoothness.

\begin{definition*}[$(L,p,R)$-Majorization Smooth]\label{def:majorant-lipschitz}
	We say that $f\in C^p(\cM)$ is $(L,p,R)$-majorization smooth on $S\subseteq\cT\cM$ for the retraction $R$ if
	\begin{equation}\label{eq:majorant-lipschitz}
		\hat{f}_{\pi(v)}(v)\leq T_p\hat{f}_{\pi(v)}(0,v)+\frac{L}{p+1}\|v\|_{\pi(v)}^{p+1}
	\end{equation}
	holds for all $v\in S$.
\end{definition*}

\noindent This definition generalizes \cite[Assumption A2]{Agarwal21} from the single $p=2$ case to arbitrary orders $p\geq1$.

The second, more technically subtle, condition generalizes derivative-based characterizations of $p$-th order Lipschitz smoothness to the Riemannian setting.

\begin{definition*}[$(L,p,R,\VT)$-Smooth]\label{def:first-order-Taylor}
	We say that $f\in C^p(\cM)$ is $(L,p,R,\VT)$-smooth on $S\subseteq\cT\cM$ if
	\begin{equation}\label{eq:first-order-Taylor}
		\left\|\left(\VT_v\right)^*D\hat{f}_{R(v)}(0)[\cdot]-T_{p-1} D\hat{f}_{\pi(v)}(0,v)[\cdot]\right\|_{op}\leq \frac{p}{p+1}\cdot L\|v\|_{\pi(v)}^p,
	\end{equation}
	holds for all $v\in S$.
\end{definition*}

\noindent This condition involves an auxiliary choice of vector transport, which is used solely for analysis and does not enter the $p$-RAR algorithm itself. Consequently, it is not a direct extension of \cite[Assumption A4]{Agarwal21}. Its inclusion enables a unified analysis covering both exponential-map smoothness and general retractions within a single framework. Introducing vector transport also broadens, rather than restricts, the class of admissible functions. For $p=1$ or $p=2$, choosing $\VT$ and $R$ to be the parallel transport and the exponential map reduces this condition to standard Lipschitz continuity of the gradient or Hessian. More generally, taking $\VT$ to be the differentiated retraction allows the assumption, and the resulting convergence theory, to apply to arbitrary retractions. Quantifying the associated smoothness constants in this setting is the focus of \cref{sec:retraction-properties}.

As the second step of the analysis, we collect several standard consequences of the AR$p$ framework that carry over to $p$-RAR, arising from the subproblem optimality conditions and $(L,p,R)$-majorization smoothness. Because their proofs follow from the same arguments as for Euclidean AR$p$ methods, except applied to the pullback functions, we omit them. These results isolate the algorithmic aspects of the analysis and involve no tangent-space comparison. They will be reused in both the first- and second-order convergence analyses. The first lemma below, a direct analogue of \cite[Lemma 2.1]{Birgin17}, follows directly from the subproblem optimality condition.

\begin{lemma}[$p$-th Order Model Sufficient Decrease]\label{lemma:model-sufficient-decrease}
	Let $x\in\cM$, $\regparam>0$, and $f\in C^p(\cM)$. If $v\in \cT_x\cM$ ensures the $p$-th order regularized model at $x$ decreases, in the sense that
	\begin{equation*}
		m_p(\regparam,x;v)\leq m_p(\regparam,x;0),
	\end{equation*}
	then $v$ ensures the sufficient decrease condition
	\begin{equation*}
		f(x)-T_p \hat{f}_x(0,v)\geq\frac{\regparam}{p+1}\|v\|_x^{p+1}
	\end{equation*}
	for the $p$-th order unregularized model. This holds for the $v_i$ computed in \eqref{eq:alg-model-decrease-condition} of $p$-RAR.
\end{lemma}

The next lemma collects two standard consequences of $(L,p,R)$-majorization smoothness: upper bounds on the regularization parameters and on the total iteration count. \newpage

\begin{lemma}[Majorization Smoothness Consequences]\label{lemma:maj-smooth-consequences}
	Let $f$ be $(L,p,R)$-majorization smooth on the set of $p$-RAR increments $\{v_i\}_{i=0}^\infty$. The following hold:
	\begin{enumerate}
		\namedpart{Regularization Parameter Bound}{lemma:sigma-bound}
		The parameters $\{\regparam_i\}_{i\geq 0}$ satisfy
		\begin{equation*}
			\regparam_i\leq \regparam_{\max}:=\max\left\{\regparam_0,\frac{\gamma_3L}{1-\eta_2}\right\}.
		\end{equation*}
		\namedpart{Increment-to-Iteration Bound}{lemma:increment-to-iteration-bound}
		If $p$-RAR has not terminated by iteration $\ell$ and there exists $B>0$ such that $\|v_i\|_{x_i}\geq B$ for every successful iteration with $i\leq \ell$, then
		\begin{equation*}
			\ell\leq\frac{(p+1)[f(x_0)-f^*]}{\eta_1\regparam_{\min}}\cdot
			\left(1-\log\nolimits_{\gamma_2}(\gamma_1)\right)\cdot\frac{1}{B^{p+1}}
			+\log\nolimits_{\gamma_2}\left(\frac{\regparam_{\max}}{\regparam_0}\right)
			=\cO\!\left(\frac{1}{B^{p+1}}\right).
		\end{equation*}
	\end{enumerate}
\end{lemma}

\noindent
Parts~\ref{lemma:sigma-bound} and~\ref{lemma:increment-to-iteration-bound} correspond respectively to \cite[Lemmas 2.2 and 2.4]{Birgin17}, with the latter consolidating the counting argument appearing in the proof of \cite[Theorem 2.5]{Birgin17}.

We now move on to the third, final, and most important step of the first-order convergence analysis, where higher-order Riemannian geometry necessitates a nontrivial extension of existing arguments. The technical hinge of this analysis is the conversion of the approximate model minimization conditions \eqref{eq:alg-model-decrease-condition} for the increment $v_i$ into a bound on the gradient norm at the next iterate. On a manifold, Taylor expansions at the current iterate and gradients at the next iterate belong to different tangent spaces. Relating these objects therefore requires controlling the effects of curvature and transport. The $(L,p,R,\VT)$-smoothness assumption provides this control at arbitrary order. We first establish this conversion in the \nameref{lemma:model-to-gradient-decrease} (\cref{lemma:model-to-gradient-decrease}), and then use it to prove our main first-order convergence result, the \nameref{thm:convergence-analysis:first-order}.

Before proving the final lemma and the main theorem, we introduce an additional assumption on $\VT$ used in the analysis of RO methods.

\begin{definition*}[Short-Step Lower Bound]\label{def:short-step-lower-bound}
	We say that the vector transport $\VT$ satisfies a $(r,b)$-short-step lower bound on $\cV\subseteq\cT\cM$ for $r,b>0$ if
	\begin{equation}\label{eq:short-step-lower-bound}
		\inf_{v\in \cV:\|v\|_{\pi(v)}\leq r}\sigma_{\min}(\VT_v)\geq b.
	\end{equation}
\end{definition*}

\noindent This holds in most practical contexts. When $\VT$ is isometric, such as when it is the parallel transport, this holds with $r=\infty$ and $b=1$. Indeed, it ought to be clear, at least on an intuitive level, that it holds on all compact subsets of a manifold, and thus for all compact manifolds for appropriate choices of $r$. It dates back to \cite[Assumption 5]{Agarwal21}, an inspiration for, and special case of, our $p$-RAR method.

We now prove the central lemma of this part.
\begin{lemma}[Increment-to-Gradient Bound]\label{lemma:model-to-gradient-decrease}
	If it holds on $\{v_i\}_{i=0}^\infty$ that $f\in C^p(\cM)$ is $(L,p,R,\VT)$-smooth and $\VT$ satisfies an $(r,b)$-short-step lower bound, then the $p$-RAR increments and gradients of $f$ at $p$-RAR's iterates are related by
	\begin{equation}\label{eq:model-to-gradient-decrease}
		\|v_i\|_{x_i}\geq\left(\frac{\sigma_{\min}(\VT_{v_i})}{L+\theta+\regparam_i}\right)^{\frac{1}{p}}\cdot \left\|\grad f(R(v_i))\right\|^{\frac{1}{p}}_{R(v_i)}
	\end{equation}
	for all $i\geq 0$ where $\|v_i\|_{x_i}\leq r$.
\end{lemma}

\begin{proof}
	Observe that $Dm_p(\regparam_i,x_i;v_i)[\cdot]=T_{p-1}D\hat{f}_{x_i}(0,v_i)[\cdot]+\regparam_i\|v_i\|_{x_i}^{p-1}\ip{v_i}{\cdot}_{x_i}$. Using the progress condition, we have
	\begin{align*}
		\theta\|v_i\|_{x_i}^p&\geq\left\|Dm_p(\regparam_i,x_i;v_i)[\cdot]\right\|_{op}\\
		&=\left\|\left(T_{p-1}D\hat{f}_{x_i}(0,v_i)[\cdot]-\left(\VT_{v_i}\right)^*D\hat{f}_{R(v_i)}(0)[\cdot]\right)+\regparam_i\|v_i\|_{x_i}^{p-1}\ip{v_i}{\cdot}_{x_i}+\left(\VT_{v_i}\right)^*D\hat{f}_{R(v_i)}(0)[\cdot]\right\|_{op}\\
		&\geq \sigma_{\min}(\VT_{v_i})\left\|\grad f(R(v_i))\right\|_{R(v_i)}-\left\|T_{p-1}D\hat{f}_{x_i}(0,v_i)[\cdot]-\left(\VT_{v_i}\right)^*D\hat{f}_{R(v_i)}(0)[\cdot]\right\|_{op}-\regparam_i\|v_i\|_{x_i}^p
	\end{align*}
	where the transition from the second to the third inequality follows from the reverse triangle inequality.
	
	Applying the $(L,p,R,\VT)$-smoothness of $f$ over $\{v_i\}_{i=0}^\infty$, this refines to
	\begin{align*}
	\theta\|v_i\|_{x_i}^p&\geq \sigma_{\min}(\VT_{v_i})\left\|\left(\grad f\right)_{R(v_i)}\right\|_{R(v_i)}-\frac{p}{p+1}\cdot L\|v_i\|_{x_i}^p-\regparam_i\|v_i\|_{x_i}^p\\
	&\geq \sigma_{\min}(\VT_{v_i})\left\|\left(\grad f\right)_{R(v_i)}\right\|_{R(v_i)}-L\|v_i\|_{x_i}^p-\regparam_i\|v_i\|_{x_i}^p.
	\end{align*}
	Consolidating all of the $\|v_i\|_{x_i}^p$ terms on the left-hand side, dividing each side by $\left(L+\theta+\regparam_i\right)$, then raising each to the power $1/p$ completes the proof.
\end{proof}

Finally, having completed the proofs of the requisite intermediate lemmas, we will prove the convergence of $p$-RAR to first-order stationary points. 

\begin{theorem}[Optimal First-Order Complexity Bound]\label{thm:convergence-analysis:first-order}
Suppose that $f$ is $(L,p,R)$-majorization smooth and $(L,p,R,\VT)$-smooth on the set of $p$-RAR increments $\{v_i\}_{i=0}^\infty$, and that $\VT$ satisfies an $(r,b)$-short-step lower bound on this set. Then $p$-RAR will find an $\epsilon_1$-stationary point in at most 
	\[
	\cO\left[\left(\frac{L}{\epsilon_1}\right)^{\frac{p+1}{p}}\right]
	\]
	total iterations and calls to the derivative oracles $\left\{x\mapsto D^k \hat{f}_x(0)\right\}_{k=0}^p$.
\end{theorem}

\begin{proof}
	It suffices to show that if $p$-RAR has not terminated by iteration $\ell$, then
	\begin{multline*}
		\ell\leq\frac{(p+1)[f(x_0)-f^*]}{\eta_1\regparam_{\min}}\cdot \left(1-\log\nolimits_{\gamma_2}(\gamma_1)\right)\cdot\max\left\{\left(\frac{L+\theta+\regparam_{\max}}{b}\right)^{\frac{p+1}{p}}\cdot \frac{1}{\epsilon_1^{\frac{p+1}{p}}}, \frac{1}{r^{p+1}}\right\}\\
		+\log\nolimits_{\gamma_2}\left(\frac{\regparam_{\max}}{\regparam_0}\right)=\cO\left[\left(\frac{L}{\epsilon_1}\right)^{\frac{p+1}{p}}\right].
	\end{multline*}
	By the \nameref{lemma:increment-to-iteration-bound} (\cref{lemma:maj-smooth-consequences}, item \ref{lemma:increment-to-iteration-bound}), it suffices to show for every successful iteration $i\leq\ell$ that
	\begin{equation}\label{eq:increment-lower-bound-first-order}
		\|v_i\|_{x_i}\geq\min\left\{
		\left(\frac{b}{L+\theta+\regparam_{\max}}\right)^{\frac{1}{p}}\cdot \epsilon_1^{\frac{1}{p}}, r\right\}.
	\end{equation}
	Unlike in the Euclidean case \cite{Birgin17}, due to the fact that $\VT$ may not act isometrically, we have to separate the successful steps into short- and long-steps. If the $i$-th iterate is short in the sense that $\|v_i\|_{x_i}<r$, then the \nameref{lemma:model-to-gradient-decrease} (\cref{lemma:model-to-gradient-decrease}), along with the fact that the algorithm has not yet terminated with a first-order stationary point, dictates that
	\[
	\|v_i\|_{x_i}\geq\left(\frac{b}{L+\theta+\regparam_i}\right)^{\frac{1}{p}}\cdot \|\grad f(R(v_i))\|^{\frac{1}{p}}_{R(v_i)}\geq \left(\frac{b}{L+\theta+\regparam_i}\right)^{\frac{1}{p}}\cdot \epsilon_1^{\frac{1}{p}}.
	\]
	Alternatively, it holds that the $i$-th iterate is long in the sense that $\|v_i\|_{x_i}\geq r$. Thus, the lower bound \eqref{eq:increment-lower-bound-first-order} must hold after bounding $\regparam_i$ above by $\regparam_{\max}$ per the \nameref{lemma:sigma-bound} (\cref{lemma:maj-smooth-consequences}, item \ref{lemma:sigma-bound}).
\end{proof}

\subsection{Second-Order Convergence Guarantees}\label{sec:second-order-guarantees}

In this subsection, we prove the second and final convergence result for the $p$-RAR method, the \nameref{thm:convergence-analysis:second-order} (\cref{thm:convergence-analysis:second-order}). The result concerns second-order stationary points and applies when $p$-RAR enforces the Second-Order Model Optimality condition \eqref{eq:alg-second-order-condition}. Specifically, we show that $p$-RAR finds a point $x\in\cM$ satisfying $\|\grad f(x)\|_x<\epsilon_1$ and $\lambda_{\min}[\hess f(x)]\ge -\epsilon_2$. Such a point is obtained in at most
\[
\cO\!\left[\max\!\left\{\left(\tfrac{1}{\epsilon_1}\right)^{\frac{p+1}{p}},\left(\tfrac{1}{\epsilon_2}\right)^{\frac{p+1}{p-1}}\right\}\right]
\]
calls to the $k\le p$ derivative oracles $x\mapsto D^k \hat{f}_x(0)$. These rates match the Euclidean optimum of \cite{Cartis17}, reaffirming this article’s foundational mantra that the complexity of $p$-th order smooth nonconvex optimization on manifolds matches its Euclidean counterpart. As a reminder, \cref{prop:pullback-to-covariant-derivative-formula} will show oracle calls to $x\mapsto D^k\hat f_x(0)$ are equivalent to oracle calls to $x\mapsto \nabla^k f_x$ under suitable conditions on $R$.

This subsection follows the structure of the first-order analysis, but requires fewer new ingredients. First, we introduce two additional assumptions, one on the retraction and one on pullback smoothness. Second, we prove an \nameref{lemma:model-to-eigenvalue-bound} (\cref{lemma:model-to-eigenvalue-bound}), which plays the same role as the \nameref{lemma:model-to-gradient-decrease} (\cref{lemma:model-to-gradient-decrease}) in the first-order setting. Third, with these pieces in place, the second-order convergence theorem follows by the same counting argument.

In this first step, we detail the additional assumptions on the retraction and objective function. The first assumption, which is standard in the literature on Riemannian second-order methods \cite[Section~5.10]{Boumal23}, ensures that the pullback Hessian coincides with the Riemannian Hessian. More formally, this assumption guarantees that $D^2\hat f_x(0_x)[\cdot,\cdot]=\ip{\hess f(x)[\cdot]}{\cdot}_x$ for all $x\in\cM$ \cite[Proposition~5.44]{Boumal23}. As a direct consequence, the spectra of $D^2\hat f_x(0_x)$ and $\hess f(x)$, viewed respectively as a bilinear form, and a self-adjoint linear map, coincide.

\begin{definition*}[Second-Order Retraction]\cite[Definition 5.42]{Boumal23}
	A retraction $R$ on $\cM$ is called second-order if the initial covariant acceleration of any retraction curve is zero, i.e. for all $s\in\cT\cM$, $\left.\frac{D}{dt}\gamma'(t)\right|_{t=0}=0$, where $\gamma(t)=R(ts)$.
\end{definition*}

\noindent This is the standard notion of a second-order retraction; see \cite[Definition~5.42]{Boumal23}.

The discovery of second-order stationary points, when $p>2$, requires a higher level of smoothness. To this end, we introduce our second assumption: a second-order variant of $(L,p,R,\VT)$-smoothness. The obvious special case of this assumption is abundant in the literature on optimal complexity for higher-order methods in Euclidean space \citetext{\citealp[Lemma 2.1]{Cartis17}; \citealp[Lemma 2.1]{Cartis20}}. Notably, this assumption is not necessary when $p=2$, as shown in \cite[Theorem 5]{Agarwal21}, but is absolutely critical when $p\geq 3$. This is due to the appearance of higher-order derivative terms in the latter setting. Its necessity to our analysis, and likely any analysis, is illustrated by \eqref{eq:Hessian-transport-equality-post-smooth} in the proof of the \nameref{lemma:model-to-eigenvalue-bound} (\cref{lemma:model-to-eigenvalue-bound}), which follows momentarily.

\begin{definition*}[$(L,p,R,\VT)$-Second-Order Smooth]\label{def:second-order-Taylor}
	We say that $f\in C^p(\cM)$ is $(L,p,R,\VT)$-second-order smooth on $\cV\subseteq\cT\cM$ if
	\begin{equation}\label{eq:second-order-Taylor}
		\left\|\left(\VT_v\right)^*D^2\hat{f}_{R(v)}(0)[\cdot,\cdot]-T_{p-2} D^2\hat{f}_{\pi(v)}(0,v)[\cdot,\cdot]\right\|_{op}\leq \frac{p(p-1)}{p+1}\cdot L\|v\|_{\pi(v)}^{p-1}, 
	\end{equation}
	holds for all $v\in\cV$. We shall refer to \eqref{eq:second-order-Taylor} as the $(L,p,R,\VT)$-second-order smoothness bound, or just second-order smoothness bound when $(L,p,R,\VT)$ are clear from context.
\end{definition*}

As the second step of the analysis, we establish the single additional lemma needed for the second-order convergence guarantee. This lemma relates the $p$-RAR increment norm to the minimum eigenvalue of the Hessian and is applied in the same manner as the \nameref{lemma:increment-to-iteration-bound} (\cref{lemma:maj-smooth-consequences}, item \ref{lemma:increment-to-iteration-bound}) in the first-order analysis; see the proof of \cref{thm:convergence-analysis:first-order}.

\begin{lemma}[Increment-to-Hessian Bound]\label{lemma:model-to-eigenvalue-bound}
	If it holds on $\{v_i\}_{i=0}^\infty$ that $f\in C^p(\cM)$ is $(L_2,p,R,\VT)$-second-order smooth, and $\VT$ satisfies the $(r,b)$-short-step lower bound, then the $p$-RAR increments and minimum eigenvalues of $f$'s Hessian at $p$-RAR's iterates are related by 
	\begin{equation}\label{eq:model-to-hessian-decrease}
		\|v_i\|_{x_i}\geq\max\left\{0,- \frac{b^2}{\theta+p(L_2+\regparam_i)}\cdot\lambda_{\min}\left[\hess f(R(v_i))\right]\right\}^{\frac{1}{p-1}}
	\end{equation}
	for all $i\geq 0$ where $\|v_i\|_{x_i}\leq r$.
\end{lemma}

\begin{proof}
Because $R$ is a second-order retraction, the bilinear form $D^2\hat{f}_{R(v_i)}(0)$ corresponds to the self-adjoint linear map $\hess f(R(v_i))$, and hence $\lambda_{\min}\!\left[\hess f(R(v_i))\right]=
\lambda_{\min}\!\left[D^2\hat{f}_{R(v_i)}(0)\right]$. Thus, it suffices to consider the case $\lambda_{\min}\!\left[\hess f(R(v_i))\right]<0$.

	Differentiating $m_p(\regparam_i,x_i;\cdot)$ twice, we have for any $\dot{v}\in \cT_{x_i}\cM$ that
	\begin{multline*}
		D^2m_p(\regparam_i,x_i;v_i)[\dot{v}^{\otimes 2}]=T_{p-2} D^2\hat{f}_{x_i}(0,v_i)[\dot{v}^{\otimes 2}]+\regparam_i\left(\|v_i\|_{x_i}^{p-1}\|\dot{v}\|_{x_i}^2+(p-1)\|v_i\|_{x_i}^{p-3}\ip{v_i}{\dot{v}}_{x_i}^2\right).
	\end{multline*}
Our arguments only concern the eigenvalues of the bilinear form, so we fix $\dot{v}$ with $\|\dot{v}\|_{x_i}=1$. Moving all terms on the right-hand side to the left, and adding $\left(\VT_{v_i}\right)^*D^2\hat{f}_{R(v_i)}(0)[\dot{v}^{\otimes 2}]$, we achieve
	\begin{multline}\label{eq:Hessian-transport-equality}
		\left(\VT_{v_i}\right)^*D^2\hat{f}_{R(v_i)}(0)[\dot{v}^{\otimes 2}]=\\
		D^2m_p(\regparam_i,x_i;v_i)[\dot{v}^{\otimes 2}]+\left(\left(\VT_{v_i}\right)^*D^2\hat{f}_{R(v_i)}(0)[\dot{v}^{\otimes 2}]-T_{p-2} D^2\hat{f}_{x_i}(0,v_i)[\dot{v}^{\otimes 2}]\right)\\
		-\regparam_i\left(\|v_i\|_{x_i}^{p-1}+(p-1)\|v_i\|_{x_i}^{p-3}\ip{v_i}{\dot{v}}_{x_i}^2\right).
	\end{multline}
	The $(L,p,R,\VT)$-second-order smoothness condition applied at $v_i$ implies
	\[
	\left(\VT_{v_i}\right)^*D^2\hat{f}_{R(v_i)}(0)[\dot{v}^{\otimes 2}]-T_{p-2} D^2\hat{f}_{x_i}(0,v_i)[\dot{v}^{\otimes 2}]\geq -\frac{p(p-1)}{p+1}\cdot L_2\|v_i\|_{x_i}^{p-1},
	\]
	which when folded into \eqref{eq:Hessian-transport-equality} yields
	\begin{multline}\label{eq:Hessian-transport-equality-post-smooth}
		\left(\VT_{v_i}\right)^*D^2\hat{f}_{R(v_i)}(0)[\dot{v}^{\otimes 2}]\geq D^2m_p(\regparam_i,x_i;v_i)[\dot{v}^{\otimes 2}]-\frac{p(p-1)}{p+1}\cdot L_2\|v_i\|_{x_i}^{p-1}\\
		-\regparam_i\left(\|v_i\|_{x_i}^{p-1}+(p-1)\|v_i\|_{x_i}^{p-3}\ip{v_i}{\dot{v}}_{x_i}^2\right).
	\end{multline}
	The Cauchy-Schwarz inequality and unit norm hypothesis for $\dot{v}$ applied to $\ip{v_i}{\dot{v}}_{x_i}^2$, together with the definition of the smallest eigenvalue of $D^2m_p(\regparam_i,x_i;v_i)[\dot{v}^{\otimes 2}]$, loosens \eqref{eq:Hessian-transport-equality-post-smooth} to
	\[
	\left(\VT_{v_i}\right)^*D^2\hat{f}_{R(v_i)}(0)[\dot{v}^{\otimes 2}]\geq \lambda_{\min}\left(D^2m_p(\regparam_i,x_i;v_i)\right)-\left(\frac{p(p-1)}{p+1}\cdot L_2+p\cdot \regparam_i \right)\|v_i\|_{x_i}^{p-1}.
	\]
	Under our assumption that $\lambda_{\min}\!\left[\hess f(R(v_i))\right]<0$, it follows from the $(r,b)$-short-step lower bound that
	\[
b^2\lambda_{\min}\left[D^2\hat{f}_{R(v_i)}(0)\right]\geq	\sigma_{\min}(\VT_{v_i})^2\cdot\lambda_{\min}\!\left[\hess f(R(v_i))\right]\geq \lambda_{\min}\left[\left(\VT_{v_i}\right)^*D^2\hat{f}_{R(v_i)}(0)\right].
	\]
	Thus we achieve
	\[
	b^2\lambda_{\min}\left[D^2\hat{f}_{R(v_i)}(0)\right]\geq \lambda_{\min}\left(D^2m_p(\regparam_i,x_i;v_i)\right)-\left(\frac{p(p-1)}{p+1}\cdot L_2+p\cdot \regparam_i \right)\|v_i\|_{x_i}^{p-1}.
	\]
	Using $p$-RAR's second-order condition \eqref{eq:alg-second-order-condition} and recognizing that $\lambda_{\min}\left[D^2\hat{f}_{R(v_i)}(0)\right]=\lambda_{\min}\left[\hess f(R(v_i))\right]$, we get
	\begin{align*}
	b^2\lambda_{\min}\left[\hess f(R(v_i))\right]=b^2\lambda_{\min}\left[D^2\hat{f}_{R(v_i)}(0)\right]&\geq -\theta\|v_i\|_{x_i}^{p-1}-\left(\frac{p(p-1)}{p+1}\cdot L_2+p\cdot \regparam_i \right)\|v_i\|_{x_i}^{p-1}\\
	&\geq -\left[\theta+p(L_2+\regparam_i)\right]\|v_i\|_{x_i}^{p-1}
	\end{align*}
	which we can rearrange to
	\[
	\|v_i\|_{x_i}^{p-1}\geq- \frac{b^2}{\theta+p(L_2+\regparam_i)}\cdot\lambda_{\min}\left[\hess f(R(v_i))\right],
	\]
	so taking the maximum of this latest expression and zero, then taking $p-1$ roots, completes the proof.
\end{proof}

As the third and final step, we combine the preceding ingredients to obtain the second-order convergence guarantee. This is embodied in the second of our main convergence theorems, which we present below.

\begin{theorem}[Optimal Second-Order Complexity Bound]\label{thm:convergence-analysis:second-order} Let $p\geq 3$. Suppose on the set of $p$-RAR increments $\{v_i\}_{i=0}^\infty$ that
	\begin{enumerate}
		\item (Functional Assumptions) $f$ is $(L_1,p,R)$-majorization smooth, $(L_1,p,R,\VT)$-smooth, and $(L_2,p,R,\VT)$-second-order smooth;
		\item (Retraction Assumptions)  $R$ is a second-order retraction; 
		\item (Transport Assumptions) and $\VT$ satisfies a $(r,b)$-short-step lower bound.
	\end{enumerate}	
	If $p$-RAR enforces the Second-Order Model Optimality condition at every iteration, then it finds an $(\epsilon_1,\epsilon_2)$-second-order stationary point in
	\[
	\cO\left[\max\left\{\left(\frac{L_1}{\epsilon_1}\right)^{\frac{p+1}{p}},\left(\frac{L_2}{\epsilon_2}\right)^{\frac{p+1}{p-1}}\right\}\right]
	\]
	total iterations and calls to the derivative oracles $\left\{x\mapsto D^k \hat{f}_x(0)\right\}_{k=0}^p$.
\end{theorem}

\begin{proof}
	
	Following the structure of the first-order analysis, it suffices to show that if $p$-RAR has not terminated by iteration $\ell+1$ then
	\begin{multline*}
		\ell\leq\frac{(p+1)[f(x_0)-f^*]}{\eta_1\regparam_{\min}}\cdot \left(1-\log\nolimits_{\gamma_2}(\gamma_1)\right)\cdot\frac{1}{B^{p+1}}\\
		+\log\nolimits_{\gamma_2}\left(\regparam_{\max}/\regparam_0\right)=\cO\left[\max\left\{\left(\frac{L_1}{\epsilon_1}\right)^{\frac{p+1}{p}},\left(\frac{L_2}{\epsilon_2}\right)^{\frac{p+1}{p-1}}\right\}\right].
	\end{multline*}

By the \nameref{lemma:increment-to-iteration-bound} (\cref{lemma:maj-smooth-consequences}, item \ref{lemma:increment-to-iteration-bound}), it suffices to show for every successful iteration $i\leq\ell$ that
	\begin{equation*}
		\|v_i\|_{x_i}\geq\min\left\{
		\left(\frac{b}{L_1+\theta+\regparam_{\max}}\right)^{\frac{1}{p}}\cdot \epsilon_1^{\frac{1}{p}}, r,	\left(\frac{b^2}{\theta+p(L_2+\regparam_i)}
		\right)^{\frac{1}{p-1}}\cdot\epsilon_2^{\frac{1}{p-1}}\right\}.
	\end{equation*}	
	If $p$-RAR has not terminated by iteration $\ell+1$, then for every successful iteration $i\le \ell$, the accepted iterate $x_{i+1}=R(v_i)$ fails at least one of the two stationarity conditions, i.e.,
	\[
	\|\grad f(x_{i+1})\|_{x_{i+1}}\ge \epsilon_1
	\quad\text{or}\quad
	\lambda_{\min}\!\left[\hess f(x_{i+1})\right]\le -\epsilon_2.
	\]
	In this case, following the same process as in the proof of the \nameref{thm:convergence-analysis:first-order}, the \nameref{lemma:model-to-gradient-decrease} (\cref{lemma:model-to-gradient-decrease}),  the \nameref{lemma:model-to-eigenvalue-bound} (\cref{lemma:model-to-eigenvalue-bound}),and the \nameref{lemma:sigma-bound} (\cref{lemma:maj-smooth-consequences}, item \ref{lemma:sigma-bound}) imply
	\begin{equation*}
		\|v_i\|_{x_i}\geq\min\left\{
		\left(\frac{b}{L_1+\theta+\regparam_{\max}}\right)^{\frac{1}{p}}\cdot \epsilon_1^{\frac{1}{p}}, r\right\}\text{ or }	\|v_i\|_{x_i}\geq 	\left( \frac{b^2}{\theta+p(L_2+\regparam_i)}
		\right)^{\frac{1}{p-1}}\cdot\epsilon_2^{\frac{1}{p-1}}.
	\end{equation*}
	This yields exactly our desired bound.
\end{proof}

\noindent In conjunction with \cite[Corollary 3]{Agarwal21}, the above convergence theorem fills in the complete second-order convergence picture for general nonconvex optimization on Riemannian manifolds. Together, these results show that second-order stationary points can be found in
\[
\cO\left[\max\left\{\left(\frac{1}{\epsilon_1}\right)^{\frac{p+1}{p}},\left(\frac{1}{\epsilon_2}\right)^{\frac{p+1}{p-1}}\right\}\right]
\]
calls to $f$'s $p$-th order derivative oracles under reasonable smoothness assumptions.

\section{Optimal-Rate Theory, Part II: Smoothness Constants}\label{sec:retraction-properties}

In this section, we quantify how the interaction of the retraction $R$ and the objective function $f$ affects the smoothness constants $L_1$ and $L_2$ in $(L_1,p,R)$-majorization smoothness, $(L_1,p,R,\VT)$-smoothness, and $(L_2,p,R,\VT)$-second-order smoothness. To ensure the $p$-RAR algorithm and our analysis depend only on $f$ and $R$, we take the vector transport to be the differentiated retraction, $dR$, throughout this section. The foundations of our analysis are a covariant \fdb formula (\cref{thm:covariant-faa-di-bruno}) and a tensorial interpretation of higher-order covariant derivatives of a retraction (\cref{lemma:TM-calculus}). Together, these yield nearly tight bounds on $L_1$ and $L_2$ when $f\in C^{p+1}(\cM)$ (Theorems \ref{thm:smoothness-bound-first-order} and \ref{thm:smoothness-bound}).

A simple Euclidean vignette, where we take $\cM=\bbR^n$ and $R$ to be any retraction, elucidates the path toward quantifying $L_1$ and $L_2$ and the precise obstacles that arise in the general Riemannian context. Standard calculus applies directly in this elementary setting: $f$ and $R$ become real- and vector-valued functions because we regard the tangent bundle of $\bbR^n$ as $\bbR^n \times \bbR^n=\bbR^{2n}$. This permits us to write $R$ as a function of two arguments, $R(x,v)$, and the $i$-th derivative of $R$ in $v$ as $D_v^i R$. Using standard arguments from \cite{Baes09,Ahmadi24,Birgin17,Cartis17} that relate Lipschitz continuity to Taylor series error bounds, and a bit of algebra, we can quantify $L_1$ and $L_2$ in terms of $f\circ R$'s derivatives. For $L_1$, specifically, we have
\[
	L_1=\sup_{(x_0,v_0)\in\bbR^{2n}}\|D^{p+1}_v\hat{f}_{x_0}(v_0)\|_{op}.
\]
A similar but more sophisticated bound on $L_2$ in terms of  $D^{p+1}_v\hat{f}_{\cdot}$ is also realizable.\footnote{
Readers interested in the minute details of why these bounds hold are referred to \cref{sec:regularity}, and specifically \cref{lemma:euclidean-lipschitz-taylor} and \cref{prop:continuity-criterion}, for details. There, we execute this program in the far more general Riemannian setting, though the majority of the technical details of these two results are essentially Euclidean.}

In any case, the critical term for both $L_1$ and $L_2$ involves $D^{p+1}_v\hat{f}_{\cdot}=D^{p+1}_v(f\circ R)$, where the function composition entangles $f$ and $R$. The higher-order chain rule, called the \fdb formula, characterizes $D^{p+1}(f\circ R)$ in terms of the derivatives of $f$ and $R$ by
\begin{equation*}
	D_v^{p+1}(f\circ R)=\sum_{k=1}^{p+1}
D^kf\circ
	B_{p+1,k}\!\left(
	D_v R,\,\dots,\,D_v^{\,p+2-k}R
	\right)
\end{equation*}
where the $B_{p+1,k}$'s are polynomials (called Bell polynomials) suitably interpreted as tuples of arguments. Using standard arguments as in \cite[Proposition 1.4.2]{Krantz02} and \cite[Proposition 3.1]{Rainer12} applied to this combinatorial expression of the derivative, one can prove the nearly tight inequality
\begin{equation}\label{eq:Gevrey-bound-Euclidean}
	\sup_{(x_0,v_0)\in \bbR^{2n}}\left\|D^{p+1}\hat{f}_{x_0}(v_0)\right\|_{op}
	\leq p!\cdot\|f\|_{p+1,\bbR^n}\cdot\|R\|_{p+1,\bbR^{2n}}\cdot\left(1+\|R\|_{p+1,\bbR^{2n}}\right)^p.
\end{equation}
Here, the semi-norm $\|h\|_{p+1,\bbR^\ell}=\max_{1\leq i\leq p+1}\sup_{s\in\bbR^{\ell}}\left[\left\|D^i h(s)\right\|_{op}/(i!)\right]$ is defined for a vector- or real-valued function $h$. Provided the right-hand side is finite, we have legitimate polynomial bounds on the smoothness constants $L_1$ and $L_2$ in terms of derivatives of $f$ and $R$. 

In principle, analogous reasoning could be performed in local coordinates on an arbitrary manifold $\cM$. However, doing so entails explosively lengthy high-order bookkeeping for $f$, $R$, and $\cM$'s geometry that obscures structure. Instead, a coordinate-free perspective lays the main question bare:
\begin{quote}
 \emph{How do we interpret the \fdb formula and higher-order covariant derivatives of $R$ with respect to $v$ on a manifold?}
 \end{quote}

Should we dare extend this to manifolds, we immediately encounter the obstacle: the mainline RO literature seemingly furnishes no general tensorial interpretation of these objects. To date, such an interpretation has only been applied for $p=2$ and to the exponential map \cite{Lezcano20}. Authors typically bypass this in two ways to find bounds on smoothness constants. Either the joint role of $R$ and $f$ is not central, so one bounds $D(f\circ R)$ without separating them, as in \cite[Assumption 3.3 and Lemma C.8]{Levin23}. Or bounds on the smoothness constants are handled implicitly through a niceness condition, as in \cite{Agarwal21}. This latter approach enjoys notable cleanliness when $p=2$, but at the cost of obscuring the higher-order differential nature at play and hindering tidy generalizability to $p \ge 3$. By introducing higher-order derivatives explicitly, we can interpret the $D_v^i R$ terms generally and derive \fdb--style bounds for pullback derivatives. 

We now develop the framework necessary to do so. Our framework is based on constructs that are remarkably underutilized in RO: the pullback connection and the Sasaki metric on $\cT\cM$. In \cref{sec:faa-di-bruno}, we introduce a \nameref{thm:covariant-faa-di-bruno} (\cref{thm:covariant-faa-di-bruno}). From this formula, we extract two consequences: a curve-wise bound on pullback derivatives and the identification of pullback and intrinsic covariant derivative oracles at the zero-section for $p$-th order retractions. In \cref{sec:covariant-calculus}, we replace curve-wise derivative expressions with clear tensorial expressions using the pullback connection and Sasaki metric. In \cref{sec:regularity}, we translate these bounds into explicit expressions for the $p$-RAR smoothness constants. Finally, in \cref{sec:comparison}, we contrast our approach with prior regularity analyses in the RO literature.

\subsection{The Covariant \fdb Formula}\label{sec:faa-di-bruno}

In this subsection, we take the first step toward expressing the smoothness constants $L_1$ and $L_2$ via bounds on $D^{p+1}\hat f_x$ styled on \eqref{eq:Gevrey-bound-Euclidean}: constructing a covariant \fdb formula. These derivatives arise from repeated differentiation of the composition of a real-valued function with a retraction curve. Thus, our goal is to develop a covariant analogue of the classical \fdb formula suitable for iterated covariant differentiation. We conclude the subsection by recording two consequences of this formula. The first, and most important, is a preliminary version of the main bound on pullback derivatives \eqref{eq:pullback-derivative-non-total}. The entirety of this subsection's sequel is dedicated to processing it into its final form. Second, we show that pullback derivative oracles coincide with higher-order covariant derivative oracles for $f$ when $R$ is a $p$-th order retraction.

The simplicity of the \fdb\ formula in Euclidean space \cite{Johnson02} and other contexts \cite{Ebrahimi14} stems from careful combinatorial accounting of where repeated derivatives of a composition are applied. The same phenomenon occurs under repeated covariant differentiation of $f \circ R$, though here the bookkeeping is inherently noncommutative. Generalized and noncommutative variants of the \fdb\ formula have a long history in combinatorics and algebra, and our combinatorial organization most closely resembles the formulation of \cite{Ebrahimi14}. As in the Euclidean setting, these combinatorial effects are absorbed into Bell-type polynomial constructions. The essential combinatorial ingredient for defining our covariant Bell polynomials is the use of ordered partitions. An ordered $k$-partition $(\sfB_1,\ldots,\sfB_k)$ is a $k$-tuple of non-empty sets that partitions $[p]:=\{1,\ldots,p\}$ and satisfies
\begin{enumerate}
	\item each block $\sfB_i$ is written in increasing order,
	$\sfB_i = \{\sfB_{i,1},\ldots,\sfB_{i,|\sfB_i|}\}$;
	\item the minimal elements of the blocks are strictly increasing: $\min(\sfB_1) < \cdots < \min(\sfB_k)$.
\end{enumerate}
We let $\ordpart(p,k)$ denote the set of ordered $k$-partitions. Note that $\ordpart(p,0)=\emptyset$ for all $p\geq 1$. Given a block $\sfB_i$ and $X_1,\ldots,X_p\in\Gamma(\cT\cM)$, we let $X_{\sfB_i}=\nabla_{X_{\sfB_{i,|\sfB_i|}}}\ldots \nabla_{X_{\sfB_{i,2}}}X_{\sfB_{i,1}}$. Given $X_1,\ldots,X_p\in\Gamma(\cT\cM)$, we call
\[
B_{p,k}^\nabla(X_1,\ldots,X_p)=\sum_{\sfB\in\ordpart(p,k)}\left(X_{\sfB_1},\ldots,X_{\sfB_k}\right)
\]
 the $(p,k)$-partial non-commutative Bell tuple. 
 
The next lemma embodies the key combinatorial property of $\ordpart$ we need for presenting and applying our \fdb formula. It can be considered a restatement of the recurrence used in deriving formulas for Stirling numbers of the second kind \cite{Stanley11}. We defer the geometry- and optimization-free proof to \cref{app:ordered-partition-properties}.

\begin{lemma}[Ordered Partition Recurrence]\label{lemma:ordered-partition-properties}
	Let $1\leq i\leq k\leq p+1$. It holds that
		\[
		\ordpart(p+1,k)=\left[\bigsqcup_{i=1}^k\insertind\left(\ordpart(p,k),i\right)\right]
		\sqcup
		\appendind\left(\ordpart(p,k-1)\right),
		\]
		where, for $\sfB\in\ordpart(p,k)$ and $\tilde{\sfB}\in\ordpart(p,k{-}1)$, we define
		\[
		\insertind(\sfB,i)
		=(\sfB_1,\ldots,\sfB_i\cup\{p{+}1\},\ldots,\sfB_k),
		\quad
		\appendind\left(\tilde{\sfB}\right)
		=\left(\tilde{\sfB}_1,\ldots,\tilde{\sfB}_{k-1},\{p+1\}\right).
		\]
\end{lemma}

Now, we can state this subsection's central result, the covariant \fdb formula. For notational clarity, we present it for tensor fields on $\cT\cM$ and in the context of the Levi-Civita connection. The same argument applies verbatim to tensor fields valued in any vector bundle endowed with a higher-order connection.

\begin{theorem}[Covariant \fdb Formula]\label{thm:covariant-faa-di-bruno}
For any tensor field $A$ and vector fields $X_1,\ldots,X_p$ on $\cM$ it holds that
	\begin{equation}\label{eq:faa-di-bruno-covariant}
			\prod_{i=1}^p\nabla_{X_{p+1-i}}A=\sum_{k=1}^p\nabla^k A\circ B_{p,k}^\nabla(X_1,\ldots,X_p).
	\end{equation}
\end{theorem}

\begin{proof}
We will proceed by induction on $p$. In the $p=1$ case, we see that $B^{\nabla}_{1,1}(X_1)=(X_1)$, so $\nabla_{X_1}A=\nabla A\circ B^{\nabla}_{1,1}(X_1)$. Now, assume the claim holds for some $p\in\bbN$, that is
	\begin{equation}\label{eq:faa-di-bruno-inductive-hypothesis}
	\prod_{i=1}^p\nabla_{X_{p+1-i}}A=\sum_{k=1}^p \nabla^k A\circ B_{p,k}^\nabla(X_1,\ldots,X_p).
	\end{equation}
	For the sake of simplicity, given $\sfB\in\ordpart(p,k)$, we will let $X_{[\sfB]}:=\left(X_{\sfB_1},\ldots,X_{\sfB_k}\right)$. Fixing such a $\sfB\in\ordpart(p,k)$ and covariantly differentiating $\nabla^k A \left(X_{[\sfB]}\right)$ in $X_{p+1}$, we compute
	\[
	\nabla_{X_{p+1}}\left[\nabla^k A \left(X_{[\sfB]}\right)\right]=\nabla^{k+1}A\left(X_{[\appendind(\sfB)]}\right)+\sum_{i=1}^k\nabla^{k} A\left(X_{[\insertind(\sfB,i)]}\right).
	\]
	Thus, covariantly differentiating both sides of the inductive hypothesis  \eqref{eq:faa-di-bruno-inductive-hypothesis} with respect to $X_{p+1}$, plugging the expression above into the resulting right-hand side, and re-indexing, we have
\[
\prod_{i=1}^{p+1}\nabla_{X_{p+2-i}}A
=\sum_{k=1}^p\left\{\sum_{\sfB\in\ordpart(p,k-1)}\nabla^kA\left(X_{[\appendind(\sfB)]}\right)+\sum_{\sfB\in\ordpart(p,k)}\sum_{i=1}^k\nabla^{k} A\left(X_{[\insertind(\sfB,i)]}\right)\right\}.
\]
Finally, applying the \nameref{lemma:ordered-partition-properties} Lemma (\cref{lemma:ordered-partition-properties}) to this final expression, we get the required equality.
\end{proof}

\noindent At first glance, the induction above may appear to be an iterated product rule. However, the combinatorics mostly arise from differentiating compositions of $A$ with curves induced by the vector fields $X_i$. Thus, it reflects a higher-order chain rule.

We now describe two key consequences of the \nameref{thm:covariant-faa-di-bruno}. The first is a general $(p{+}1)$-order bound on pullback derivatives that makes explicit how derivatives of the retraction enter the smoothness constants. Once combined with our later expressions for higher-order covariant derivatives of $R$, it yields all pullback function derivative bounds used in the paper (\cref{cor:pullback-function-derivative-formula-total}). To state this result, we introduce the following notation: the relation $\boldL=(\ell_1,\ldots,\ell_k)\vDash_k p$ indicates that $\ell_1,\ldots,\ell_k$ are $k$ natural numbers summing to $p$, and we write $\binom{p}{\boldL}=\binom{p}{\ell_1\ldots\ell_k}$. The proof is a non-trivial combinatorial exercise, but it is devoid of geometry and optimization, so we leave it to \cref{app:pullback-function-derivative-formula-non-total}.

\begin{corollary}[Curve-wise Pullback Function Derivative Bound]\label{cor:pullback-function-derivative-formula-non-total}
	If $f\in C^{p+1}(\cM)$, and $R$ is a retraction on $\cM$, then for all $v,\tilde{v}\in\cT\cM$ with $\pi(v)=\pi(\tilde{v})$,
		\begin{equation}\label{eq:pullback-derivative-non-total}
\left\|D^{p+1}\hat{f}_{\pi(v)}(v)\left[\tilde{v}^{\otimes (p+1)}\right]\right\|\leq (p+1)!\cdot\max_{1\leq k\leq p+1}\frac{\|(\nabla^k f)_{R(v)}\|_{op}}{k!}\cdot \sum_{k=1}^{p+1} \sum_{\ell\vDash_k p+1}\prod_{i=1}^k \frac{ \left\|D_t^{\ell_i-1}\gamma'(0)\right\|}{\ell_i!},
	\end{equation}
	where the covariant derivatives are taken along the curve $\gamma(t):=R(v+t\tilde{v})$.
\end{corollary}

Our second consequence is the fulfillment of our promise that, for $p$-th order retractions, the pullback derivatives and covariant derivatives coincide along the zero-section. Informally, a $p$-th order retraction is one whose second through $p$-th derivatives vanish. Formally, we have the following definition.

\begin{definition*}
	A retraction $R:\cT\cM\to\cM$ is called $p$-th order if $\left.\frac{D^{i-1}}{dt^{i-1}}\gamma'(t)\right|_{t=0}=0$ for all $v\in\cT\cM$ and $2\leq i\leq p$ where $\gamma(t)=R(tv)$.
\end{definition*}

\noindent The next proposition codifies our claim concerning the covariant derivatives of a function on $\cM$ and the derivatives of its pullback.

\begin{proposition}\label{prop:pullback-to-covariant-derivative-formula}
	Let $f\in C^p(\cM)$ and $R:\cT\cM\to\cM$ be a $p$-th order retraction. For all $1\leq i\leq p$ and $x\in\cM$, it holds that $D^i\hat{f}_x(0_x)=\Sym\left(\nabla^i f_x\right)$, where the latter is the symmetrization of $\nabla^i f_x$.
\end{proposition}

\begin{proof}
	Fix $x\in\cM$ and $v\in\cT_x\cM$. In the \nameref{thm:covariant-faa-di-bruno}, take $A=f$ and $X=\gamma'(t)$ with $\gamma(t):=R(tv)$. By the $p$-order hypothesis, all terms in the formula \eqref{eq:faa-di-bruno-covariant} of the \nameref{thm:covariant-faa-di-bruno} disappear for  $k<i$. Thus,  \eqref{eq:faa-di-bruno-covariant} reduces to $D^i\hat{f}_x(0_x)\left[v^{\otimes i}\right]=(\nabla^i f)_x\left[v^{\otimes i}\right]=\Sym\left[\left(\nabla^i f\right)_x\right]\left[v^{\otimes i}\right]$. Hence, the proposed equality holds because polarization implies symmetric tensors are equal when they agree for diagonal evaluations.
\end{proof}

\subsection{Simpler Derivative Formulas via Double Tangent Bundle Calculus}\label{sec:covariant-calculus}

In this subsection, we convert the preliminary pullback derivative bound \eqref{eq:pullback-derivative-non-total} into its final form \eqref{eq:p+1-derivative-bound} by expressing it in terms of higher-order covariant derivatives of $R$. This requires making sense of such higher-order derivatives and reinterpreting retraction curve covariant derivatives in tensorial form. Informally speaking, we are searching for a notion of a higher-order covariant derivative ``$\nabla^k R_v"$ that makes sense of
\begin{equation}\label{eq:informal-curve-wise-to-total}
	\left.\frac{D^{k-1}}{dt^{k-1}}\gamma'(t)\right|_{t=0}
	\;=\;
	\text{``}\nabla^k R_v\text{''}\left(\tilde{v}^{\otimes k}\right),
\end{equation}
where $\gamma(t)=R(v+t\tilde{v})$ with $v,\tilde{v}\in\cT\cM$ satisfying $\pi(v)=\pi(\tilde{v})$. The difficulty in making this interpretation precise is that $dR$ is not a vector or tensor field on a single manifold, but a smooth map from $\cT\cT\cM$ into $\cT\cM$. Since covariant derivatives apply only to vector and tensor fields over a fixed base manifold, this precludes a naïve application of the Levi-Civita connection to $dR$.

To resolve this, we require two ingredients. First, we introduce the pullback connection, which provides a notion of covariant differentiation for tensor fields along a smooth map. Second, we endow $\cT\cM$ with the Sasaki metric, the canonical Riemannian geometry on the tangent bundle induced by the Riemannian structure of $\cM$. This choice yields simple expressions for higher-order covariant derivatives of $R$. We develop these two components in this order.

We begin with the first ingredient, the pullback connection, which provides a notion of covariant differentiation for tensor fields along a smooth map. Our presentation largely follows \cite{Lee09} notationally.\footnote{Readers familiar with pullback connections will notice that we define it using the pullback bundle. Our definition frames the connection via its universal property, rather than an explicit commutative diagram, for technical simplicity.} While the pullback connection is standard in differential geometry, it has seen only limited explicit use in RO analyses, where retraction regularity is typically handled through curve-wise arguments. Only one other paper uses it to study retraction regularity, and it focuses specifically on the exponential map and $p=2$ \cite{Lezcano20}. Like the ordinary Levi--Civita connection, covariant differentiation along a map is formulated in terms of sections of appropriate tangent vector bundles.
\needspace{1\baselineskip}
\begin{definition*}
Let $F : \cM \to \cN$ be a smooth map between smooth manifolds.  We call
\begin{enumerate}
	\item  A smooth function $Y:\cM\to\cT\cN$ a (smooth) section of $\cT\cN$ along $F$ if $\pi_{\cN} \circ Y = F$. 
	\item A smooth function $T:\cM\to \cT^{(0,k)}\cM \otimes \cT\cN$ a ($k$-covariant) tensor field along $F$ if for each $x\in\cM$, $T_x:=T(x):(\cT_x\cM)^{\otimes k} \to \cT_{F(x)}\cN$ is $k$-linear.
\end{enumerate}
We let $\Gamma_F(\cT\cN)$ denote the set of all sections of $\cT\cN$ along $F$, and $\Gamma_F(\cT\cM^{(0,k)} \otimes \cT\cN)$ the set of all $k$-covariant tensor fields along $F$.
\end{definition*}

\noindent With this piece in place, we can define the pullback connection.

\begin{definition*}[Pullback Connection]\cite[Definition 12.7 \& Theorem 12.15]{Lee09}
	Let $F:\cM\to\cN$ be smooth. The pullback of $\nabla^\cN$ by $F$ is the unique map $\nabla^F:\Gamma(\cM)\times \Gamma_F(\cT\cN)\to \Gamma_F(\cT\cN)$ such that 
	\[
	\nabla^F_X(\tilde{Y}\circ F)
	=
	\left(\nabla^{\cN}_{dF(X)}\tilde{Y}\right)\circ F
	\]
	for all $\tilde{Y}\in\Gamma(\cN)$ and $X\in\Gamma(\cM)$.
\end{definition*}

\noindent By construction, the pullback connection is $C^\infty(\cM)$-linear in its first argument, $\bbR$-linear in its second, and satisfies the product rule $\nabla^F_X(fY)=X(f)Y+f\nabla^F_XY$ for all $f\in C^\infty(\cM)$ and $Y\in\Gamma_F(\cT\cN)$. RO sophisticates are already well-acquainted with an omnipresent instance of this object in the literature: the covariant derivative along a curve. Taking $\cM=I$ to be an interval, and $\gamma:I\to\cN$ to be a smooth curve, the covariant derivative along $\gamma$ is precisely $\nabla^{\gamma}$.

Covariantly differentiating $F$, and in particular $dF$, requires extending the pullback connection $\nabla^F$ to tensor fields along $F$. 
Given the Levi-Civita connections $\nabla^{\cM}$ and $\nabla^{\cN}$ on $\cM$ and $\cN$, and a $k$-covariant tensor field 
$T\in\Gamma_F(\cT\cM^{(0,k)}\otimes\cT\cN)$, we define 
$\nabla^F T\in\Gamma_F(\cT\cM^{(0,k+1)}\otimes\cT\cN)$ by
\begin{multline}\label{eq:higher-order-covariant-derivative}
	(\nabla^F T)(X_1,\ldots,X_{k+1})
	=\nabla^{F}_{X_{k+1}}\!\left[T(X_1,\ldots,X_k)\right]\\
	-\sum_{i=1}^k 
	T\!\left[\left(\otimes_{\ell=1}^{i-1}X_\ell\right)
	\otimes \nabla^{\cM}_{X_{k+1}}X_i
	\otimes\left(\otimes_{\ell=i+1}^k X_\ell\right)\right],
\end{multline}
where $X_1,\ldots,X_{k+1}$ are (possibly local) sections of $\cT\cM$ \cite[Section~12.3]{Lee09}.

\begin{definition*}
	Let $F:\cM\to\cN$ be smooth. The first covariant derivative of $F$ is defined to be $dF$. For $k\geq 2$, the $k$-th covariant derivative of $F$ is defined as $\left(\nabla^F\right)^{k-1} dF$.
\end{definition*}

Two obstacles remain to interpreting \eqref{eq:informal-curve-wise-to-total} in terms of higher-order covariant derivatives of $R$. First, $(\nabla^R)^{k-1}dR$ is a $\cT\cM$-valued $k$-tensor field along $R$, whose arguments lie in $\cT\cT\cM$ not in $\cT\cM$. This is resolved by the vertical lift, which canonically lifts tangent vectors into $\cT\cT\cM$. For $v\in\cT\cM$ and $\dot{v}\in\cT_{\pi(v)}\cM$, the vertical lift is defined by
\[
\vl_v(\dot{v})
=
\left.\frac{d}{dt}(v+t\dot{v})\right|_{t=0},
\]
where $t\mapsto v+t\dot{v}$ is viewed as a curve in the vector space $\cT_{\pi(v)}\cM$. Note that if $X$ is a (local) vector field on $\cM$, then $v\in\cT\cM\mapsto \vl_v(X_{\pi(v)})$ defines a local vector field on $\cT\cM$ which we denote $\vl(X)$.

The second obstacle is the presence of the summation terms in \eqref{eq:higher-order-covariant-derivative}, which involve covariant derivatives of the input vector fields. Making sense of \eqref{eq:informal-curve-wise-to-total} therefore requires a choice of geometry on $\cT\cM$ for which these terms vanish when the $X_i$ are suitable extensions of vertical lifts of curves $v+t\tilde{v}$. A good geometry enforces this cancellation; a bad one does not. In concrete terms, the summation terms in \eqref{eq:higher-order-covariant-derivative} disappear precisely when the input vector fields are parallel along the differentiation directions. This observation motivates the following general identity, which isolates the effect of parallelism on higher-order covariant derivatives.

\begin{proposition}[Parallel Section Tensor Derivative Formula]\label{thm:parallel-section-covariant-calculus}
	Let $F:\cM\to\cN$ be smooth, and $T$ be a $k$-times continuously differentiable $r$-tensor field along $F$. 
	Let $X_1,\ldots,X_r$ and $\dot X_1,\ldots,\dot X_k$ be local vector fields near $x\in\cM$ that are parallel along each $\dot X_j$. Then
\begin{equation}\label{eq:covariant-derivative-map-parallel-section}
	\left(\nabla^{F}\right)^k T_x\left[\left(\otimes_{i=1}^r X_i\right)\otimes\left(\otimes_{j=1}^k \dot{X}_j\right)\right]=\left(\prod_{j=1}^k \nabla^F_{\dot{X}_{k+1-j}}\right)T_x\left(\otimes_{i=1}^r X_i\right).
\end{equation}
\end{proposition}

\begin{proof}
Each term of the form $\nabla_{\dot{X}_\ell} X_i=\nabla_{\dot{X}_\ell} \dot{X}_j=0$ disappears in \eqref{eq:higher-order-covariant-derivative}. Thus, equation \eqref{eq:covariant-derivative-map-parallel-section} follows immediately from induction applied to \eqref{eq:higher-order-covariant-derivative}. 
\end{proof}

The ideal geometry on $\cT\cM$ for our purposes is its most prominent and, as Boumal convincingly argues in a recent blog post \cite{Boumal25_Sasaki}, its most natural one: the Sasaki metric, which remains underutilized in the RO literature. We require only two properties of this metric, and we recall its basic structure solely for completeness. At a high level, the Sasaki metric decomposes each tangent space $\cT_v\cT\cM$ into an orthogonal sum of the vertical subspace $\Ker(d\pi^{\cT\cM}_{v})$ and a complementary subspace, both canonically isometric to $\cT_{\pi(v)}\cM$. The first property of this geometry that we require is that the vertical lift map $\vl_v$ is an isometric embedding of $\cT_{\pi(v)}\cM$ onto $\Ker(d\pi^{\cT\cM}_v)$. The second property is that vertical lifts satisfy a strong parallelism condition, recorded in the following lemma.

\begin{lemma}[Vertical Lift Parallelism]\label{lemma:vertical-lift-parallel}\cite[p.~28]{Albuquerque19}
	If $X$ and $Y$ are smooth vector fields on a neighborhood of $x\in\cM$, then for all $v\in\cT_x\cM$ we have $\nabla_{\vl_v(X)} \vl(Y)=0$.
\end{lemma}

Together, these properties allow curve-wise derivatives to be expressed tensorially in terms of higher-order covariant derivatives of $R$, as formalized in the next theorem. Using it we can immediately give the final form of the \fdb formula for the higher-order derivatives of the pullback functions.

\begin{theorem}[Vertical Tensor Derivative Formula]\label{lemma:TM-calculus}
 Let $T$ be a $k$-times continuously differentiable $r$-tensor along a smooth map $F:\cT\cM\to\cN$. For $x\in\cM$ and $v,v_1,\ldots,v_r,\dot{v}_1,\ldots,\dot{v}_k\in\cT_x\cM$,  it holds that
	\begin{equation}\label{eq:covariant-derivative-map-vertical-curve}
		\left(\nabla^{F}\right)^k T_v\left[\left(\otimes_{i=1}^r \vl_v(v_i)\right)\otimes\left(\otimes_{j=1}^k \vl_v(\dot{v}_j)\right)\right]=\left.\prod_{j=1}^k\frac{D^{\gamma_{k+1-j}}}{dt_{k+1-j}} T_\cdot\left[\left(\otimes_{i=1}^r \vl_\cdot(v_i)\right)\right]\right|_{\substack{t_i=\hat{t}_i\\ i=1,\ldots,k}}
	\end{equation}
	where the covariant derivatives on the right are along the curves $\gamma_i(t_i):=F\left(v+\sum_{j=1}^{i-1}\hat{t}_j \dot{v}_j+t_i \dot{v}_i\right)$. In particular, if $R$ is a retraction on $\cM$, then for any $v,\tilde{v}\in\cT\cM$ with $\pi(v)=\pi(\tilde{v})$ it holds that
	\begin{equation*}
		\left.\frac{D^{k-1}}{dt^{k-1}}\gamma'(t)\right|_{t=0}
		\;=\;
\nabla^{k-1} dR_v \left[\vl_v\left(\tilde{v}^{\otimes k}\right)\right],
	\end{equation*}
	where $\gamma(t)=R(v+t\tilde{v})$.
\end{theorem}

\begin{proof}
The retraction-based equation follows immediately from \eqref{eq:covariant-derivative-map-vertical-curve}, so it suffices to prove that claim. Fix a coordinate chart $U$ containing $x$. We can extend any $\tilde{v}\in \cT_x\cM$ to a local vector field via the rule $y\in U\mapsto (y,\tilde{v})$ in the induced coordinates for $\cT\cM$ over $\pi^{-1}(U)$. Let $V$, $V_1$, \ldots, $V_r$, $\dot{V}_1$, \ldots, $\dot{V}_k\in\cT_x\cM$ be the extensions of our tangent vectors at hand. As stated in our introduction to the covariant derivative along a smooth map, the classic instance of this operator class is the covariant derivative along a smooth curve. Symbolically, this manifests to our benefit as $\frac{D^{\gamma_j}}{dt_j}=\nabla^\cN_{\dot{v}_j}$ for $1\leq j\leq k$. Thus, the claim follows from the \nameref{thm:parallel-section-covariant-calculus} (\cref{thm:parallel-section-covariant-calculus}), the \nameref{lemma:vertical-lift-parallel} Lemma (\cref{lemma:vertical-lift-parallel}), and the observation that
	\begin{align*}
	\left(\nabla^{F}\right)^k T_v\left[\left(\otimes_{i=1}^r \vl_v(v_i)\right)\otimes\left(\otimes_{j=1}^k \vl_v(\dot{v}_j)\right)\right]&=\left.\prod_{i=1}^k\nabla^\cN_{\dot{v}_{k+1-i}}T_\cdot\left[\otimes_{j=1}^r \vl_{\cdot}(V_j)\right]\right|_{v}\\
&=\left.\prod_{j=1}^k\frac{D^{\gamma_{k+1-j}}}{dt_{k+1-j}} T_\cdot\left[\left(\otimes_{i=1}^r \vl_\cdot(v_i)\right)\right]\right|_{\substack{t_i=\hat{t}_i\\ i=1,\ldots,k}}.
	\end{align*}
\end{proof}

Now, we can present our final formulation of the pullback derivative bound. We will use two convenient derivative semi-norms that mimic those in the Euclidean bound \eqref{eq:Gevrey-bound-Euclidean}. Given $\cX\subseteq\cM$ and $\cV\subseteq\cT\cM$, and $k\in\bbN$, we let 
\begin{equation*}
	\|f\|_{k,\cX}=\max_{1\leq i\leq k}\sup_{x\in\cX}\frac{\left\|\nabla^{i}f_x\right\|_{op}}{\left(i!\right)}\quad\text{ and }\quad\|R\|_{k,\cV}=\max_{1\leq i\leq k}\sup_{v\in\cV}\frac{\left\|\left(\nabla^R\right)^{(i-1)}dR_v\right\|_{op}}{\left(i!\right)}.
\end{equation*}
In terms of these norms, our bound is cleanly displayed in the following corollary.
\begin{corollary}[Pullback Function Derivative Bound]\label{cor:pullback-function-derivative-formula-total}
	If $f\in C^{p+1}(\cM)$, $R$ is a retraction on $\cM$, and $\cV\subseteq\cT\cM$, then
\begin{equation}\label{eq:p+1-derivative-bound}
	\sup_{v\in \cV}\left\|D^{p+1}\hat{f}_{\pi(v)}(v)\right\|_{op}
	\leq (p+1)!\cdot\|f\|_{p+1,R(\cV)}\cdot\|R\|_{p+1,\cV}\cdot\left(1+\|R\|_{p+1,\cV}\right)^p.
\end{equation}
\end{corollary}
 
\noindent Like its preliminary version in \cref{cor:pullback-function-derivative-formula-non-total}, the proof is essentially a combinatorial exercise. The hard work toward it was completed above, and geometry only enters through the substitution $D_t^{\ell_i-1} \gamma'(0)=(\nabla^R)^{\ell_i-1} dR_v\left[\vl_v(\tilde{v})^{\otimes (\ell_i-1)}\right]$. For this reason, we leave the proof to \cref{app:pullback-function-derivative-formula-total}.

\subsection{Joint Bounds on Higher-Order Smoothness Constants}\label{sec:regularity}

In this penultimate subsection, we fulfill our promise of simple and insightful bounds on the $p$-RAR smoothness constants $L_1$ and $L_2$, based on the \nameref{cor:pullback-function-derivative-formula-total}. We first partially translate bounds on the $(p+1)$-order derivatives of the pullback functions into bounds on the smoothness constants (\cref{prop:continuity-criterion}). This immediately yields bounds on the first-order constant $L_1$ (\cref{thm:smoothness-bound-first-order}). Next, we use newly established properties of higher-order covariant derivatives of retractions (\cref{thm:retraction-derivative-properties}) to further process these bounds into one on the second-order smoothness constant $L_2$ (\cref{thm:smoothness-bound}). Finally, we show that all resulting bounds are finite when the $p$-RAR iterate set is compactly contained in $\cM$ (\cref{lemma:iterate-compactness}).

The translation begins with the following lemma, standard in the literature on higher-order methods, which converts derivative bounds into smoothness bounds \cite{Baes09,Ahmadi24,Birgin17,Cartis17}.

\begin{lemma}[Euclidean Lipschitz-to-Smoothness]\label{lemma:euclidean-lipschitz-taylor}
	Let $C$ be star-convex at $0$ in a finite-dimensional inner product space $V$, and $f:C\to\bbR$ be a $(p+1)$-continuously differentiable function such that $D^p f$ is $L$-Lipschitz at $0$ in the sense that
	\[
	\|D^p f(v)-D^pf(0)\|_{op}\leq (p+1)!\cdot L\|v\|,
	\]
	for all $v\in C$. Then the following hold 
\begin{subequations} \begin{gather} f(v)\leq T_p f(0,v)+\frac{L}{p+1}\|v\|^{p+1}\\ \left\|Df(v)[\cdot]-T_{p-1} Df(0,v)[\cdot]\right\|_{op}\leq \frac{p}{p+1}\cdot L\|v\|^p\\ \left\|D^2f(v)[\cdot,\cdot]-T_{p-2} D^2f(0,v)[\cdot,\cdot]\right\|_{op}\leq \frac{p(p-1)}{p+1}\cdot L\|v\|^{p-1}\label{eq:euclidean-lipschitz-taylor} \end{gather} \end{subequations}
	for all $v\in C$. In particular, all of the aforementioned inequalities hold if $\|D^{p+1} f(v)\|_{op}$ is bounded above by $L$ for all $v\in C$.
\end{lemma}

\noindent To adapt this lemma to the Riemannian setting, we need one minor additional assumption on the subset of $\cT\cM$ under consideration. This assumption is decidedly unrestrictive, because it can always be assumed to hold in practice.

\begin{definition*}
	We say that $\cV\subseteq\cT\cM$ is zero-section star-convex if $tv\in\cV$ for all $t\in[0,1]$ and $v\in\cV$.
\end{definition*}

\noindent With this assumption and the previous lemma, we transform bounds on pullback function derivatives into bounds on first-order smoothness and proto-bounds for second-order smoothness.

\begin{lemma}[Pullback-based Lipschitz-to-Smoothness]\label{prop:continuity-criterion}
	Let $f\in C^p(\cM)$, $R:\cT\cM\to\cM$ be a retraction, and $\cV\subseteq\cT\cM$ be zero-section star-convex. If for all $v\in \cV$, 
	\begin{equation}\label{eq:continuity-criterion}
		\left\|D^p\hat{f}_{\pi(v)}(v)[\cdot]-D^p\hat{f}_{\pi(v)}(0)[\cdot]\right\|_{op}\leq (p+1)!\cdot L\|v\|_{\pi(v)},
	\end{equation}
	then $f$ is $(L,p,R)$-majorization smooth, $(L,p,R,dR)$-smooth, and satisfies the proto-$(L,p,R,dR)$-second-order smoothness bound
	\begin{multline}\label{eq:proto-second-order-smooth}
		\left\|(dR)_v^* \circ D^2\hat{f}_{R(v)}(0)[\cdot,\cdot]-T_{p-2} D^2\hat{f}_{\pi(v)}(0,v)[\cdot,\cdot]\right\|_{op}\leq \frac{p(p-1)}{p+1}\cdot L\|v\|_{\pi(v)}^{p-1}\\
		+\|f\|_{p+1,R(\cV)}\cdot\left\|\nabla^R dR_v\right\|_{op}, 
	\end{multline}
	for all $v\in\cV$.
\end{lemma}

\begin{proof}
	Fix $x\in\pi(\cV)$. The \nameref{lemma:euclidean-lipschitz-taylor} Lemma (\cref{lemma:euclidean-lipschitz-taylor}) applied to $\hat{f}_x$ immediately implies that $(L,p,R)$-majorization smoothness and $(L,p,R,dR)$-smoothness hold. We also have that
\begin{equation*}
			\left\|D^2\hat{f}_{\pi(v)}(v)[\cdot,\cdot]-T_{p-2} D^2\hat{f}_{\pi(v)}(0,v)[\cdot,\cdot]\right\|_{op}\leq p(p-1)\cdot L\|v\|_{\pi(v)}^{p-1},
	\end{equation*}
	which is close to, but not exactly, the proto-$(L,p,R,dR)$-second-order smoothness of \eqref{eq:proto-second-order-smooth}. Observe by the \nameref{thm:covariant-faa-di-bruno}, \nameref{lemma:TM-calculus}, and symmetrization that
	\begin{align*}
		D^2\hat{f}_{R(v)}(v)[\cdot,\cdot]&=\vl_v^*\circ (dR)_v^* \circ\nabla^2\hat{f}_{R(v)}(0)[\cdot,\cdot]+\nabla f_{R(v)}\left\{\vl_v^*\circ\Sym[\nabla^R dR_v](\cdot,\cdot)\right\},
	\end{align*}
	so the triangle inequality gives  \eqref{eq:proto-second-order-smooth}, but with $\left\|\nabla f_{R(v)}\left\{\Sym[\nabla^R dR_v](\cdot,\cdot)\right\}\right\|_{op}$ in place of\\ $ \|f\|_{p+1,R(\cV)}\cdot\left\|\nabla^R dR_v\right\|_{op}$. However, the latter term is an upper bound on the former.
\end{proof}

\noindent This lemma and the \nameref{cor:pullback-function-derivative-formula-total} imply an immediate bound on $L_1$.

\begin{theorem}[Smoothness Bounds: First-Order]\label{thm:smoothness-bound-first-order}
Let $R:\cT\cM\to\cM$ be a retraction, and $\cV\subseteq\cT\cM$ be zero-section star-convex. If $f\in C^{p+1}(\cM)$ then it is $(L_1,p,R)$-majorization smooth and $(L_1,p,R,dR)$-smooth with
\[
L_1=\|f\|_{p+1,R(\cV)}\cdot\|R\|_{p+1,\cV}\cdot\left(1+\|R\|_{p+1,\cV}\right)^p.
\]
\end{theorem}

The final obstacle to bounding $L_2$ is the nuisance term $\|f\|_{p+1,R(\cV)}\cdot\|\nabla^R dR_v\|_{op}$ in \eqref{eq:proto-second-order-smooth}, which is not of the form $\tilde{L}\|v\|^{p-1}$ for some $\tilde{L}\geq 0$. Resolving this requires a deeper analysis of the geometric properties enjoyed by $p$-th order retractions, which also underpins our future comparison with prior pullback regularity results. These properties enforce vanishing of $R$'s higher covariant derivatives at the zero-section, with curvature effects appearing only beyond order $p$. This explains why higher-order retractions locally mimic Euclidean behavior to the degree required by higher-order optimization methods. 

To establish these properties, we need two tools. First, we recall that the curvature tensor $\curvTens$ associated with $\nabla$ is defined for $X,Y\in\Gamma(\cM)$ by $\curvTens(X,Y)Z=\nabla_X\nabla_YZ-\nabla_Y\nabla_XZ-\nabla_{[X,Y]}Z$. For $x\in\cM$, the map $\curvTens(X,Y)|_x$ defines a linear endomorphism of $\cT_x\cM$ and depends only on the values of $X,Y,Z$ at $x$. The second tool is a version of Taylor's theorem with integral remainder for smooth vector fields along smooth curves. It is a variant of a theorem from \cite{Rempala88}.

\begin{proposition}[Vector Field Taylor's Theorem]\label{lemma:Taylor-vector-field}
	If $\gamma:[0,1]\to\cM$ is a smooth curve and $V$ is a smooth vector field along it, then
	\begin{equation}\label{eq:Taylor-vector-field}
		\partrans_{t \to 0}V(t)=V(0)+\sum_{i=1}^k\frac{t^i}{i!}\frac{D^iV}{dt^i}(0)+\frac{1}{k!}\int_0^t (t-s)^k\;\partrans_{s \to 0}\left[\frac{D^{k+1}V}{dt^{k+1}}(s)\right]\;ds
	\end{equation}
	holds for $k\in\bbN$ and $t\in[0,1]$. Consequently,
	\begin{equation}\label{eq:Taylor-vector-field-bound}
		\left\|\partrans_{t \to 0}V(t)-\left\{V(0)+\sum_{i=1}^k\frac{t^i}{i!}\frac{D^iV}{dt^i}(0)\right\}\right\|\leq \frac{1}{(k+1)!}\cdot\sup_{u\in[0,1]}\left\|\frac{D^{k+1}V}{dt^{k+1}}(u)\right\|.
	\end{equation}
\end{proposition}

\noindent We defer the proof of this proposition to the appendix (\cref{app:Taylor-vector-field}).

\begin{theorem}[$p$-th Order Retraction Derivative Properties]\label{thm:retraction-derivative-properties}
	Let  $R:\cT\cM\to\cM$ be a $p$-th order retraction, and $1\leq i\leq p$. The following hold:
	\begin{enumerate}
		\item ($p$-th Order Zero-Section Vanishing) For all $x\in\cM$ and $v_1,\ldots,v_i\in\cT_x\cM$, it holds that
		\[
		\left(\nabla^R\right)^{i-1} dR_{0_x}\left(\otimes_{j=1}^i\vl_{0_x}(v_j)\right)=0.
		\]
		\item ($p$-th Order Bounds) If $\cV\subseteq\cT\cM$ is zero-section star-convex, then it holds for all $v\in\cV$ that
		\begin{equation}\label{eq:retraction-p-order-bound}
			\left\|(\nabla^R)^{i-1} dR_v\right\|_{op}\leq \min_{i< k\leq p+1}\left[\frac{1}{(k+1-i)!}\cdot\left(\sup_{\tilde{v}\in\cV}\left\|(\nabla^R)^{k-1} dR_{\tilde{v}}\right\|_{op}\right)\cdot\|v\|^{k-i}\right].
		\end{equation} 
	\end{enumerate}
\end{theorem}

\begin{proof}
	Throughout, we let $V,V_1,\ldots,V_i$ be local extensions of $v,v_1,\ldots, v_i$ as in the proof of the \nameref{lemma:TM-calculus} (\cref{lemma:TM-calculus}). To ease notation, let $\tilde{V}=\vl(V)$ and $\tilde{V}_j=\vl(V_j)$ for $1\leq j\leq i$.\\
	 
	\noindent 1. It holds for all $w\in\cT_x\cM$ that $(\nabla^R)^{i-1} dR_{0_x}(\vl_{0_x}(w))=0$ by the $p$-th order property. Thus, by polarization, it suffices to show that $\vl_{0_x}^*\circ(\nabla^R)^{i-1}dR_{0_x}$ is symmetric.

	We argue by induction on $i$, using the induction hypotheses at orders $i-1$ and $i-2$. The $i=1$ case holds trivially because linear operators are symmetric. In the case $i=2$, we must show that
	\[
	\left.\nabla^R dR\left(\tilde{V}_1,\tilde{V}_2\right)\right|_{0_x}-	\left.\nabla^R dR\left(\tilde{V}_2,\tilde{V}_1\right)\right|_{0_x}=0.
	\]
	Using the fact that the Levi-Civita connection is torsion-free, and the definition of the pullback connection,  we can rewrite this in terms of Lie brackets as
\begin{equation}\label{eq:retraction-symmetric-base-case}
		\left.\nabla^R dR\left(\tilde{V}_1,\tilde{V}_2\right)\right|_{0_x}-	\left.\nabla^R dR\left(\tilde{V}_2,\tilde{V}_1\right)\right|_{0_x}
		=\left.\left[dR(\tilde{V}_1),dR(\tilde{V}_2)\right]\right|_{0_x}-\left.dR\left(\left[\tilde{V}_1,\tilde{V}_2\right]\right)\right|_{0_x}.
	\end{equation}
It holds that, because $R$ is a local diffeomorphism at $0_x$ and the Lie bracket is natural with respect to pushforwards,
\[
\left.\left[dR(\tilde{V}_1),dR(\tilde{V}_2)\right]\right|_{0_x}
=\left.dR\left(\left[\tilde{V}_1,\tilde{V}_2\right]\right)\right|_{0_x}.
\]
Thus, \eqref{eq:retraction-symmetric-base-case} vanishes.

	Now, assume for some $2\leq k\leq p-1$ that the claim holds for $k-1$ and $k-2$. It suffices to show that the $(k+1)$-th covariant derivative of $R$ is symmetric in the last two arguments and equals zero, i.e. 
	\[
	\left.\left(\nabla^R\right)^k dR\left(\otimes_{i=1}^{k-2}\tilde{V}_i,\tilde{V}_{k-1},\tilde{V}_k\right)\right|_{0_x}=\left.\left(\nabla^R\right)^k dR\left(\otimes_{i=1}^{k-2}\tilde{V}_i,\tilde{V}_k,\tilde{V}_{k-1}\right)\right|_{0_x}=0.
	\]
	From the fact that all sections are parallel along each $\tilde{V}_i$, we have from the definition of the higher-order derivative that
\begin{multline*}
		\left.\left(\nabla^R\right)^k dR\left(\otimes_{i=1}^{k-2}\tilde{V}_i,\tilde{V}_{k-1},\tilde{V}_k\right)\right|_{0_x}-\left.\left(\nabla^R\right)^k dR\left(\otimes_{i=1}^{k-2}v_i,\tilde{V}_k,\tilde{V}_{k-1}\right)\right|_{0_x}\\
		=\left.\curvTens(V_{k-1},V_k)\left[\left(\nabla^R\right)^{k-2} dR\left(\otimes_{i=1}^{k-2}\tilde{V}_i\right)\right]\right|_{0_x}.
\end{multline*}
	The function $Z\mapsto[\curvTens(V_{k-1},V_k)Z]_{0_x}$ only depends upon the value of $Z$ at $0_x$. In our case,\\ $\left.\left(\nabla^R\right)^{k-2}dR\left[\otimes_{i=1}^{k-1}\vl(V_i)\right]\right|_{0_x}=0$ so the right-hand side equals zero at $0_x$.\\
	
	\noindent 2. Fix $i<k\leq p+1$. Applying \eqref{eq:Taylor-vector-field-bound} of the \nameref{lemma:Taylor-vector-field} to $\left(\nabla^R\right)^{i-1} dR\left(\otimes_{j=1}^i\tilde{V}_j\right)$ in conjunction with the first item of this theorem and the \nameref{lemma:TM-calculus} (\cref{lemma:TM-calculus}), we arrive at
\[
		\left\|\left.\left(\nabla^R\right)^{i-1} dR\left(\otimes_{j=1}^i\tilde{V}_j\right)\right|_v\right\|\leq\frac{1}{(k+1-i)!}\cdot\sup_{t\in[0,1]}\left\|\left.\left(\nabla^R\right)^k dR\left(\otimes_{j=1}^i\tilde{V}_j,\tilde{V}^{\otimes (k-i)}\right)\right|_{tv}\right\|.
\]
	We discarded the parallel transport on the left-hand side of \eqref{eq:Taylor-vector-field-bound} above using the fact it acts isometrically. Thus, it holds by zero-section star-convexity of $\cV$ that
\[
		\left\|\left(\nabla^R\right)^{i-1} dR_v\left(\otimes_{j=1}^i\vl_{v}(v_j)\right)\right\|\leq \frac{1}{(k+1-i)!}\cdot\sup_{\hat{v}\in\cV}\left\|\left(\nabla^R\right)^k dR_{\hat{v}}\right\|_{op}\cdot\prod_{j=1}^i\|v_j\|\cdot\|v\|^{k-i}.
\]
	Taking the supremum over $v_1,\ldots,v_i$ with unit norm and the minimum over $i< k\leq p+1$ gives the claim for $\left(\nabla^R\right)^{i-1} dR_{v}$.
\end{proof}

Using these retraction properties, we can now give our bound on $L_2$.

\begin{theorem}[Smoothness Bounds: Second-Order]\label{thm:smoothness-bound}
	
	Let $\cV\subseteq\cT\cM$ be zero-section star-convex. If $f\in C^{p+1}(\cM)$ and $R$ is a $p$-th order retraction, then $f$ is $(L_2,p,R,dR)$-second-order smooth with
	\begin{align*}
				L_2&=\|f\|_{p+1,R(\cV)}\cdot\|R\|_{p+1,\cV}\cdot\left[\left(1+\|R\|_{p+1,\cV}\right)^p+\frac{(p+1)^2}{p(p-1)}\right].
	\end{align*}
\end{theorem}

\begin{proof}
Using the second item of the \nameref{thm:retraction-derivative-properties}, we can bound the nuisance term in the proto-second-order smoothness inequality \eqref{eq:proto-second-order-smooth} according to
\begin{align*}
\|f\|_{p+1,R(\cV)}\cdot\left\|\nabla^R dR_v\right\|&\leq \frac{1}{p!}\cdot\|f\|_{p+1,R(\cV)}\cdot \left\|\left(\nabla^R \right)^p dR_v\right\|_{op} \cdot\|v\|^{p-1}\\
	&\leq (p+1)\cdot\|f\|_{p+1,R(\cV)}\cdot\|R\|_{p+1,\cV} \cdot\|v\|^{p-1}
\end{align*}
which clinches the formula.
\end{proof}

Notably, all of our expressions depend upon the set $\cV\subseteq\cT\cM$. For our purposes, $\cV$ can be set to the zero-section star-convex and closed hull of the set of $p$-RAR increments $\{v_i\}_{i=1}^\infty$, i.e. $\cV:=\overline{[0,1]\{v_i\}_{i=1}^\infty}$ where $[0,1]\{v_i\}_{i=1}^\infty:=\{tv_i:i\geq 1,t\in[0,1]\}$. If the set of increments is unbounded, then the quantities in our formulas could be potentially infinite. Even on compact manifolds, such as Stiefel and Grassmannian manifolds, it is not a priori clear that the set of increments will live in a compact set, because tangent bundles are not compact. Thankfully, if the $p$-RAR iterates are pre-compact, as on compact manifolds, then the increments also lie in a pre-compact set by the following proposition. This guarantees our $L_1$ and $L_2$ are finite. 

\begin{proposition}[$p$-RAR Increment Compactness]\label{lemma:iterate-compactness} Let $f\in C^{p+1}(\cM)$, $R$ be a $p$-th order retraction, and $\{v_i\}_{i=1}^\infty$ be the set of $p$-RAR increments. If $p$-RAR's iterate sequence is compactly contained in $\cM$, then the following hold:
	\begin{enumerate}
		\item The star-convex and closed hull of the increment sequence, $\cV:=\overline{[0,1]\{v_i\}_{i=1}^\infty}$, is compact in $\cT\cM$.
		\item The set $R(\cV)$ is compactly contained in $\cM$.
	\end{enumerate}
Consequently, $\|f\|_{p+1,\cV}$ and $\|R\|_{p+1,\cV}$ are finite, By extension, so are the $L_1$ and $L_2$ of (Theorems \ref{thm:smoothness-bound-first-order} and \ref{thm:smoothness-bound}).
\end{proposition}

\noindent We defer the proof of this technical proposition to the appendix (\cref{app:iterate-compactness}).

\subsection{Comparison with Prior Analyses}\label{sec:comparison}

We complete this section by comparing our bounds to those existing in the literature. Specifically, we focus on the common style found in the RO literature, epitomized by \cite{Agarwal21} and \cite{Criscitiello23}. Typically, in verifying \eqref{eq:continuity-criterion} of the \nameref{prop:continuity-criterion} Lemma (\cref{prop:continuity-criterion}) for $p=2$, authors bound the norm of
\begin{equation}\label{eq:continuity-criterion-difference}
D^2\hat{f}_x(v)\left[\cdot,\cdot\right]-D^2\hat{f}_x(0)\left[\cdot,\cdot\right]
\end{equation}
for all $v\in \cT_x\cM$ and $x\in\cM$. 

The premise of the analysis is to decompose \eqref{eq:continuity-criterion-difference} into a sum of controllable quantities defined via covariant differentiation along curves of the form $\gamma_v(\tilde v;t) := R(v + t\tilde v)$ for $v,\tilde v \in \cT_x\cM$. The paper \cite{Agarwal21} decomposes \eqref{eq:continuity-criterion-difference} as
\begin{multline}\label{eq:agarwal-decomposition}
	D^2\hat{f}_{R(v)}(v) - D^2\hat{f}_{\pi(v)}(0)
	=
	\left\{\ip{\hess f(R(v))[\gamma_v'(\cdot;0)]}{\gamma_v'(\cdot;0)} - \ip{\hess f(x)[\cdot]}{\cdot}\right\}+\ip{W_v[\cdot]}{\cdot}.
\end{multline}
Here, $\gamma_v'(\tilde v;t)$ and $\gamma_v''(\tilde v;t)$ denote the first and second covariant derivatives of the retraction curve. The operator $W_v$ is a symmetric bilinear form involving the gradient of $f$ and retraction-induced acceleration terms. The decomposition of \cite{Criscitiello23} follows the same philosophy, but splits the expression into four grouped terms instead of two. In both papers, each term is bounded separately using curve-wise estimates involving derivatives of $f$ and $R$, some of which quietly contain higher-order covariant derivatives.

However, this direct decomposition-based analysis of \eqref{eq:continuity-criterion-difference} entails two significant trade-offs. First, these analyses are natural and effective at low order, e.g., $p=2$, but may become difficult to organize systematically at higher order. Iterated curve-wise covariant differentiation introduces substantial combinatorial bookkeeping, often requiring ad-hoc regularity conditions and constants for both $f$ and the retraction $R$. By contrast, our analysis of the $(p+1)$-pullback function derivative via the covariant \fdb formula and intrinsic covariant derivatives of the retraction cleanly characterizes the underlying combinatorics.

Second, curve-wise bounds on the terms in \eqref{eq:agarwal-decomposition} and its variant in \cite{Criscitiello23} often require either an explicit specification of the retraction or additional structural conditions on it. In \cite{Criscitiello23}, the analysis assumes that $R$ is the manifold’s exponential map. Alternatively, \cite{Agarwal21} assumes that the retraction is second-order nice, a seemingly stronger condition than second-order retraction accuracy. Below, we introduce a notion of niceness for arbitrary order $p \ge 2$ which generalizes the $p=2$ case.

\begin{definition*}[$p$-th Order Nice Retraction]\label{def:p-order-niceness} Let $p\geq 2$. We say that a retraction $R$ on $\cM$ is $p$-th order nice on $\cV\subseteq \cT\cM$ if there exist constants $c_0,\ldots,c_p$ such that the following hold for all $v\in \cV$ and $\dot{v}\in\cT_{\pi(v)}\cM$:
	\begin{enumerate}
		\item For all $t\in[0,1]$, $\|\gamma'_{tv}(\dot{v};0)\|\leq c_0\|\dot{v}\|$
		\item For all $t\in[0,1]$, $\left\|\left.\frac{D}{ds}\gamma'_{sv}(\dot{v};0)\right|_{s=t}\right\|\leq c_1\|v\|\|\dot{v}\|$
		\item For $2\leq i\leq p$, $\left\|\left.\frac{D^{i-1}}{dt^{i-1}} \gamma'_v(\dot{v};t)\right|_{t=0}\right\|\leq c_i\|v\|\|\dot{v}\|^i$.
	\end{enumerate}
\end{definition*}

The niceness constants, being framed in terms of covariant derivatives taken along curves, obscure the role of the higher-order covariant derivatives of $R$.
As a consequence, verifying that these constants are finite typically requires ad-hoc, retraction-specific analyses, even on compact manifolds.
This perspective motivates the conjecture in \cite{Agarwal21} that most, but possibly not all, retractions are nice on compact manifolds.
Reframing the niceness constants in terms of covariant derivatives of $R$ shows that this assumption is not, in fact, stronger for purposes of $p$-RAR, including the $p=2$ setting of \cite{Agarwal21}.

\begin{corollary}\label{cor:p-order-nice-verification}
	If $R$ is a $p$-th order retraction, then it is $p$-th order nice on compact subsets of $\cT\cM$. In particular, it is nice on the set of $p$-RAR iterates if $\cM$ is compact. 
\end{corollary}

\begin{proof}
Let $\cV\subseteq\cT\cM$ be any compact and zero-section star-convex subset containing the set of $p$-RAR iterates. Let $\radius(\cV)$ be the radius of $\cV$, i.e. $\radius(\cV)=\sup_{v\in\cV}\|v\|_{\pi(v)}$, which is finite by compactness of $\cM$ and \cref{lemma:iterate-compactness}. It is clear that $c_0=\sup_{v\in\cV}\|dR_v\|_{op}$ and from the second item of the \nameref{thm:retraction-derivative-properties} (\cref{thm:retraction-derivative-properties}) that for $1\leq i\leq p$
\[
c_i=\frac{1}{2}\cdot\sup_{v\in\cV}\left\|(\nabla^R)^i dR_v\right\|_{op}\cdot \radius(\cV).
\]
\end{proof}

\section{Toward Implementation, Part I: Practical Retractions on Two Manifolds}\label{sec:retractions-practical}

In this section, we study the polar- and QR-decomposition-based retractions of \cite{Gawlik18} on the Stiefel and Grassmannian manifolds. We show that these retractions are fit for use within the $p$-RAR framework. Specifically, we prove that they are well-defined on the entire tangent bundle and achieve retraction order $p$ for appropriate choices of $p$.

To start our review of the retractions presented in \cite{Gawlik18}, we recall two key matrix decompositions. Given a full-rank matrix $X \in \bbR^{n \times p}$, there exists a pair of matrices $(P,H)$, called the polar decomposition of $X$, such that $X = PH$, $P \in \St(n,p)$, and $H \in \bbR^{p \times p}$ is positive semi-definite. The QR decomposition of $X$ is a pair of matrices $(Q,R)$ such that $X = QR$, $Q \in \St(n,p)$, and $R \in \bbR^{p \times p}$ is upper triangular. The operators $\cP : \bbR^{n \times p} \to \bbR^{n \times p}$ and $\cQ : \bbR^{n \times p} \to \bbR^{n \times p}$ that map $X$ to its polar factor $P = \cP(X)$ and Q-factor $Q = \cQ(X)$ are smooth on the set of full column-rank matrices.

In terms of these operators, \cite{Gawlik18} proposes the retraction classes
\begin{align}
	V\in\cT_X\St(n,p)&\mapsto R_X^{\St,m}(V):=X\cP\left(\Theta_m(X^\top V)\right)\label{eq:Gawlik-Stiefel-retraction}\\
	V\in\cT_{[X]_{\Gr}}\Gr(n,p)&\mapsto R_{[X]_{\Gr}}^{\Gr,m}(V):=\left[X\cQ\left(\Theta_m(X^\top V)\right)\right]_{\Gr}\label{eq:Gawlik-Grassmannian-retraction}
\end{align}
for any $X \in \St(n,p)$, where $[\cdot]_{\Gr}$ denotes the canonical projection from $\St(n,p)$ to $\Gr(n,p)$. Here,
\[
\Theta_m(H) := \sum_{k=0}^{m} \binom{m}{k} \frac{(2n - k)!}{(2n)!} (2H)^k,
\]
for $H \in \Skew(p)$. The analysis of \cite{Gawlik18} focuses on establishing the distance estimates
\begin{equation}\label{eq:Gawlik-bounds}
	\begin{aligned}
		\| R_X^{\St,m}(tV) - \Exp_X^{\St}(tV) \| &= \cO(t^{m+1}), \\
		d_{\Gr}\!\left(R_X^{\Gr,m}(tV), \Exp_X^{\Gr}(tV)\right) &= \cO(t^{2m+1}),
	\end{aligned}
\end{equation}
where the first bound is expressed in the Euclidean norm on $\bbR^{n \times p}$ and the second in the Riemannian distance on $\Gr(n,p)$. These results can be viewed as defining an alternative, approximation-based notion of retraction order. The question of global well-definedness on the entire tangent bundle is left open.

This leaves two small but essential barriers to certifying these retractions for use in the $p$-RAR framework.  First, we must establish that the retractions are well-defined on the entire tangent bundle. We defer the proof of this well-definedness, stated in the following proposition, to the appendix (\cref{app:retractions-practical}).

\begin{proposition}\label{lemma:well-defined-retraction}
	The retractions \eqref{eq:Gawlik-Stiefel-retraction} and \eqref{eq:Gawlik-Grassmannian-retraction} are well-defined on the entire tangent bundle of their respective manifolds.
\end{proposition}

Second, we must reconcile the approximation-based notion of $p$-th order retraction proposed in \cite{Gawlik18} with the natural RO definition. It is easiest and clearest to perform this analysis in normal coordinates. To this end, we need a characterization of higher-order covariant derivatives of curves in normal coordinates. Thankfully, the covariant and ordinary derivatives of a curve agree at $0$ in normal coordinates. That is, if $x \in \cM$ and $\gamma$ is a smooth curve passing through $x$, then in a normal coordinate neighborhood centered at $x$ it holds that
\begin{equation}\label{eq:normal-coordinate-derivative}
	\left.\left(\nabla_{\gamma'}\right)^k \gamma'\right|_x
	=
	\frac{d^{k+1}}{dt^{k+1}} \tilde{\gamma}(0),
\end{equation}
for all $k \ge 0$. This fact is well-known \cite{Stack1,Stack2} and dates back to at least \cite{Tsirulev95}. We will also need to translate the bounds in \eqref{eq:Gawlik-bounds} to bounds in normal coordinates. This is the content of the next lemma.

\begin{lemma}\label{lemma:distance-bound-normal-embedded}
	Let $\cM$ and $\cN$ be Riemannian manifolds, with $\cN$ isometrically embedded in $\cM$. For any $x\in\cN$, there exists a normal neighborhood $\Exp^{\cN}_x(U)\subseteq\cN$ centered  at $x$, and $C_0,C_1>0$ such that
	\begin{equation}\label{eq:distance-bound-normal-embedded}
		\left\|\left(\Exp_x^{\cN}\right)^{-1}(y)-\left(\Exp_x^{\cN}\right)^{-1}(z)\right\|_x\leq C_0 d_{\cN}\left(y,z\right) \leq C_1 d_{\cM}\left(y,z\right) 
	\end{equation}
	for all $y,z\in \Exp^{\cN}_x(U)$.
\end{lemma}

\begin{proof}
	Both inequalties pivot on a single simple Lipschitz inequality for smooth maps. Suppose $V$ is a geodesically convex subset of a smooth manifold $\cM_1$ and $F:\cM_1\to\cM_2$ is a smooth map between Riemannian manifolds. Let$\|DF|_V\|:=\sup_{y\in V}\|DF_y\|_{op}$. Then for any $x_0,x_1\in V$ it holds that
	\begin{equation}\label{eq:C1-Lipschitz-bound}
		d_{\cM_2}\left(F(x_0),F(x_1)\right)\leq\int_0^1 \|F\circ \gamma'(t)\|dt\leq \|DF|_V\|\cdot \int_0^1 \|\gamma'(t)\|dt=\|DF|_V\|\cdot d_{\cM_1}(x_0,x_1),
	\end{equation}
	where $\gamma:[0,1]\to V$ is the unique geodesic segment connecting $x_0$ and $x_1$.
	
	Next, we will define the $U$ in the lemma statement. We will let $\iota_{\cN}$ be the inclusion of $\cN$ in $\cM$, an isometric embedding, and $\cP_{\cN}$ be the metric projection from $\cM$ to $\cN$. The projection $\cP_{\cN}$ is smooth on an open subset of $\cM$ that contains $\cN$ \cite{Foote84}. Selecting $r>0$ sufficiently small we may assume that $\overline{B_{\cN}(x,r)}$ is geodesically convex, and the restriction of $\Exp^{\cN}_x$ is a local diffeomorphism on $\left(\Exp^{\cN}_x\right)^{-1}\left[\overline{B_{\cN}(x,r)}\right]$. Shrinking $r$ even further, by continuity of $\iota_{\cN}$ and $\cP_{\cN}$, we may assume that $B_{\cN}(x,r)$ is contained in an open, geodesically convex subset $\hat{V}$ of $\cM$ on which $\cP_{\cN}$ is smooth. We will let $U=\left(\Exp^{\cN}_x\right)^{-1}\left[B_{\cN}(x,r)\right]$, and endow $U$ with the metric induced by the Riemannian metric $\ip{\cdot}{\cdot}_x$ at $x$.
	
	The proof setup, the hard part of the proof, is done; the inequalities will flow quickly from two applications of \eqref{eq:C1-Lipschitz-bound}. For the left-most inequality in \eqref{eq:distance-bound-normal-embedded}, we set $V=U=\cM_1$, $\cM_2=\cN$, and $F=\left(\Exp^{\cN}_x\right)^{-1}$ in \eqref{eq:C1-Lipschitz-bound} to conclude
	\[
	\left\|\left(\Exp_x^{\cN}\right)^{-1}(y)-\left(\Exp_x^{\cN}\right)^{-1}(z)\right\|_x\leq \left\|\left(\Exp^{\cN}_x\right)^{-1}|_U\right\| \cdot d_{\cN}\left(y,z\right) 
	\]
	for all $y,z\in\Exp_x^{\cN}(U)=B_{\cN}(x,r)$. Thus, we let $C_0=\left\|\left(\Exp^{\cN}_x\right)^{-1}|_U\right\|$. For the right-most inequality, we set $\cM_1=\cM$, $\cM_2=\cN$, $V=\hat{V}$, and $F=\cP_{\cN}$, and apply the inclusion $\Exp^{\cN}_x(U)=B_{\cN}(x,r)\subseteq\hat{V}$, to conclude
	\[
	d_{\cN}(y,z)=d_{\cN}(\cP_{\cN}(y),\cP_{\cN}(z))\leq \|\cP_{\cN}|_{\hat{V}}\|\cdot d_{\cM}(y,z)
	\]
	for all $y,z\in \Exp^{\cN}_x(U)$.
\end{proof}

Using the lemma above, we can now prove that the retractions of \cite{Gawlik18} have appropriate order. This theorem generalizes \cite[Proposition 3]{Absil12} to arbitrary manifolds and retraction orders greater than two.

\begin{theorem}\label{thm:p-th-order-Gawlik}
	Let $\cM$ and $\cN$ be Riemannian manifolds, with $\cN$ isometrically embedded in $\cM$. If $R:\cT\cN\to\cN$ satisfies $R(0_x)=x$ for all $x\in\cN$, and 
	\begin{equation}\label{eq:p-th-order-criterion}
		d_\cM\left(R(ts),\exp^{\cN}(ts)\right)=\cO(t^{p+1}),
	\end{equation}
	for all $s\in\cT\cN$ with $t$ sufficiently small, then $R$ is a $p$-th order retraction. Consequently, $R^{\St,m}$ and $R^{\Gr,m}$ are $m$-th and $(2m+1)$-th order retractions.
\end{theorem}

\begin{proof}
	The claim concerning the retractions of \cite{Gawlik18} will follow readily once we prove the main body of the theorem. It is already known from \eqref{eq:Gawlik-bounds} that the Stiefel retractions abide by \eqref{eq:p-th-order-criterion} with $p=m$ and $\cM=\St$. It follows from \eqref{eq:Gawlik-bounds} and the right-most inequality of \eqref{eq:distance-bound-normal-embedded} of \cref{lemma:distance-bound-normal-embedded} that \eqref{eq:p-th-order-criterion} holds for the Grassmannian retractions with $p=2m+1$ and $\cM=\Gr$.
	
	Let $\Exp^{\cN}_x(U)$ be the normal neighborhood of $x$ from \cref{lemma:distance-bound-normal-embedded}. For $t_0>0$ sufficiently small it holds for all $t\in[0,t_0]$ that $R(ts),\Exp_x^{\cN}(ts)\in \Exp^{\cN}_x(U)$. By extension, for all $t\in[0,t_0]$ it holds that $\left\|\tilde{R}(t)-ts\right\|_x=\cO(t^{p+1})$ where $\tilde{R}(t)=\left(\Exp_x^{\cN}\right)^{-1}\left[R(ts)\right]$. The standard Euclidean version of Taylor's theorem with remainder applied in $U$ to $\tilde{R}(t)$ shows that $\tilde{R}^{(1)}(0)=s$, and $\tilde{R}^{(i)}(0)=0$ for $2\leq i\leq p$. By \eqref{eq:normal-coordinate-derivative}, it then follows for $2\leq i\leq p$ that $R$ is a $p$-th order retraction.
\end{proof}

\section{Toward Implementation, Part II: Solving the $p=3$ Subproblem}\label{sec:subproblem-solver}

In this section, we overcome the final obstacle to implementing $3$-RAR and, more broadly, third-order RO methods: efficiently solving the third-order regularized subproblem. In Euclidean space, this is the only remaining barrier to the practical use of third-order methods and drives a substantial recent literature \cite{Ahmadi24,Cartis25_CQR,Cartis25_QQR,Cartis25_DTM}. Having resolved all other impediments to implementation on manifolds, efficient subproblem solution is likewise the remaining roadblock for RO. However, subproblem solution on manifolds presents an additional difficulty absent in Euclidean space: for many manifolds, tangent spaces do not admit readily available coordinate bases. Second-order Riemannian methods such as \cite{Agarwal21} provide a natural blueprint for resolving this issue by adapting the joint Krylov--secular equation approach of \cite{Cartis11a}. For third-order models, no analogous framework exists, because available tensor Krylov methods \cite{Savas13} are not tailored to symmetric third-order Taylor models or third-order regularized subproblem solvers.

The $3$-RAR subproblem consists of minimizing a quartically regularized third-order Taylor model over a tangent space. Our contribution is a Krylov-based framework supplying the basis-construction component of a joint Krylov--secular equation approach. At present, no such Krylov--secular equation framework exists even in Euclidean space for third-order regularized models, although \cite{Cartis25_CQR,Cartis25_QQR,Cartis25_DTM} identify this direction as a promising avenue. The resulting reduced problems are then solved by joining this construction with any modern Euclidean solver for quartically regularized third-order models. Concretely, the subproblems we consider are quartically regularized third-order polynomials, that is, functions $M:V\to\bbR$, defined on a finite-dimensional inner-product space $V$, of the form
\begin{equation}\label{eq:regularized-third-order-polynomial}
	M(v)=c+\ip{g}{v}+\frac{1}{2}\ip{v}{Hv}+\frac{1}{6}\ip{v}{T(v,v)}+\frac{\sigma}{4}\|v\|^4.
\end{equation}
Here, $c\in\bbR$, $g\in V$, and $H$ is self-adjoint and $(u,v,w)\mapsto \ip{u}{T(v,w)}$ is symmetric.

The remainder of this short section is laid out as follows. In \cref{sec:Krylov-iteration}, we describe our adaptation of the maximal Krylov iteration from \cite{Savas13}. Then, in \cref{sec:Krylov-framework}, we describe our Krylov-based framework for minimizing functions of the form \eqref{eq:regularized-third-order-polynomial}.

\subsection{The Symmetric Maximal Krylov Iteration}\label{sec:Krylov-iteration}
\enlargethispage{\baselineskip}

In this subsection, we describe an adaptation of the Krylov iteration of \cite{Savas13}, called the \nameref{alg:Krylov-iteration-sym} and abbreviated as \KrylovSymShort, that is tailored to symmetric third-order tensors. This iteration can be combined within our subproblem solution framework with any selected Euclidean solver for quartically regularized third-order subproblems. The principal obstacle to directly applying the maximal Krylov iteration of \cite{Savas13} is that it generates three separate bases, one for each argument of $\ip{\cdot}{T(\cdot,\cdot)}$, to account for possible asymmetry. By contrast, our subproblem minimization requires a single basis shared across all arguments. In this sense, the iterations of \cite{Savas13} are generalizations of the Arnoldi iteration to tensors, whereas our subproblem setup calls for a Lanczos-type iteration. To this end, we modify the procedures of \cite{Savas13} to generate a single basis when $\ip{\cdot}{T(\cdot,\cdot)}$ is symmetric.

Our Krylov iteration incrementally constructs an orthonormal basis and an associated coordinate representation of $T$. Provided no benign breakdown occurs, the scheme generates a full basis of $V$ in at most $\dim(V)$ steps. Hence, the iteration yields a complete coordinate representation of $T$ with respect to the constructed basis. As in matrix Krylov methods, benign breakdowns occur when newly generated directions are linearly dependent on the existing orthonormal set. In such cases, we adopt standard remedies for Krylov methods, which augment the existing basis with a randomly generated orthogonal vector to restore linear independence. 

For clarity of presentation, breakdown detection and remediation are not shown explicitly in the pseudocode. We also do not study structural properties of the coordinate representation of $T$, such as the upper Hessenberg-type structure established for the original maximal Krylov iteration in \cite[Theorem~5]{Savas13}. Our focus here is on basis construction for subproblem solution; questions concerning sparsity and finer structural properties are left for future study.\vspace{.5em}

\begin{algorithm}[H]
	\nametag{Symmetric Maximal Krylov Iteration}
	\caption{Symmetric Maximal Krylov Iteration $(\KrylovSym)$}
	\label{alg:Krylov-iteration-sym}
	\small
	\KwData{Orthonormal set $U_r=\{u_i\}_{i=1}^r$, symmetric bilinear $T:V\times V\to V$, coordinate representation $\tilde{T}$ of $T$ on $U_{r_-}$ (plus partial $U_r$ evaluations) with $r_-\leq r$}
	$r_+=r$\;
\For{$1 \leq i \leq j \leq r$}{
	\If{\textnormal{pair }$(i,j)$ \textnormal{has not been used before}}{
		$r_+=r_+ +1$\;
		$\hat{u}_{r_+}=T(u_j,u_i)$\;
		$u_{r_+}=\frac{\left(\hat{u}_{r_+}-\sum_{k=1}^{r-1}  \ip{u_k}{\hat{u}_{r_+}}u_k\right)}{\left\|\hat{u}_{r_+}-\sum_{k=1}^{r-1}  \ip{u_k}{\hat{u}_{r_+}}u_k\right\|}$\;
		\For{\textnormal{each permutation} $(\delta,\beta,\zeta)$ \textnormal{of} $(\ell,i,j)$ \textnormal{such that }$1\leq \ell\leq r_+$}{
			$\tilde{T}_{\delta\beta\zeta}:=\ip{u_\ell}{\hat{u}_{r_+}}$\;
		}
	}
}
\KwResult{Updated orthonormal set $U_{r_+}=\{u_i\}_{i=1}^{r_+}$ and coordinate representation $\tilde{T}$ (plus partial $U_{r_+}$ evaluations)  on $U_r$ }
\end{algorithm}

\vspace{.5em} Notably, \KrylovSymShort\ operates on a coordinate representation of $T$ on $U_{r_-}$ that also includes partial evaluations $\langle u_\ell,T(u_i,u_j)\rangle$ where at least one vector lies in $U_r \setminus U_{r_-}$. Such evaluations are generated in lines~7--8 of Algorithm~\ref{alg:Krylov-iteration-sym} during basis expansion. Accordingly, the method is designed to accept and return tensor representations containing these partial entries. These entries are retained and reused in subsequent iterations to avoid recomputation.

\subsection{Krylov Hybrid Minimization Framework}\label{sec:Krylov-framework}
\enlargethispage{\baselineskip}

In this subsection, we describe the \nameref{alg:Krylov-Framework} (\cref{alg:Krylov-Framework}), which combines Krylov basis construction with a Euclidean solver for polynomials $M$ of the form \eqref{eq:regularized-third-order-polynomial}. At each iteration, the framework expands its Krylov basis, then forms and approximately minimizes a reduced model $\tilde M_r$ by restricting $M$ to the current Krylov subspace.

The framework imposes only a minimal requirement on the Euclidean subproblem solver, which we denote $\solve$. At each iteration, $\solve$ is only required to return a point $s^*$ satisfying the termination conditions given by $\|\grad \tilde{M}_r(s^*)\| \leq \theta \|s^*\|^3$ and $\tilde{M}_r(s^*) \leq c$ on the current Krylov subspace. Our Krylov iteration generates a complete basis of $V$ in finitely many steps in the absence of benign breakdowns, so the generated subspace eventually equals $V$. Because the termination conditions are evaluated for the full model $M$, these requirements are sufficient to ensure that the framework eventually produces a solution of the $3$-RAR subproblem. In particular, premature termination on a strict Krylov subspace cannot occur unless the full-space termination conditions are already satisfied.\vspace{.5em}

\begin{algorithm}[H]
	\caption{Krylov Hybrid Minimization Framework}
	\nametag{Krylov Hybrid Minimization Framework}
	\label{alg:Krylov-Framework}
	\small
	\KwData{$H:V\to V$ symmetric linear, $T:V^2\to V$ symmetric bilinear, $g\in V\setminus\{0\}$, Euclidean solver $\solve$, $c\in\bbR$, $\theta>0$}
	$U_1=\{u_1:=g/\|g\|\}$\;
	$n:=\dim(V)$, $\tilde{T}=[\ ]$, $\tilde{H}=[\ ]$\; 
	$r_{-}=0$, $r=1$\;
	\Repeat{$\left\|\grad M\left(u^*\right)\right\|\leq\theta \left\|u^*\right\|^3$ and $M(u^*)\leq c$}{
		Get $U_{r_+}$, $r_+$, and $\tilde{T}$ from $\KrylovSym$\;
		Update $\tilde{H}$'s entries: if $\tilde{H}_{ij}$ empty for $1\leq i\leq j\leq r$, then $\tilde{H}_{ij},\tilde{H}_{ji}=\ip{u_i}{Hu_j}$\;
		Compute $\left(s^*,\tilde{M}_r(s^*),\left\|\grad \tilde{M}_r(s^*)\right\|^2\right)=\solve\left(\tilde{M}_r(s):=M\left(\sum_{i=1}^r s_i u_i\right)\right)$\;
		Let $u^*=\sum_{i=1}^r s_i^* u_i$\;
		$r_{-}=r$\;
		$r=r_+$\;
	}
	\KwResult{$u^*=\sum_{i=1}^r s_i^* u_i$}
\end{algorithm}

\vspace{.5em} The line-by-line operation of the above framework can be described at a high level as follows. At each iteration, we assume an orthonormal subset $U_r=\{u_i\}_{i=1}^r$ of $V$, together with coordinate representations $\tilde{T}$ and $\tilde{H}$ of $T$ and $H$ on the smaller subspace spanned by $U_{r_-}$. In addition to fully known entries, $\tilde{T}$ includes partial evaluations $\langle u_\ell, T(u_i,u_j)\rangle$ where at least one vector lies in $U_r \setminus U_{r_-}$. These evaluations are generated during the construction of $U_r$ and are retained to avoid recomputation in later iterations. First, in line~5, the framework expands $U_r$ to $U_{r_+}$ and updates $\tilde{T}$ to yield a complete coordinate representation on $U_r$, with partial evaluations for new vectors. Second, in line~6, the framework updates the coordinate representation of $H$ on the current basis as needed. Third, in line~7, the framework invokes a Euclidean solver, denoted by $\solve$, to approximately minimize the reduced quartically regularized third-order model over the current Krylov subspace. Finally, in lines~8--11, the framework forms the candidate step $u^*$, updates the bookkeeping indices, and checks the termination condition.

\section{Toward Implementation, Part III: Numerical Experiments}\label{sec:numerics}

In this section, we present the results of numerical experiments for several implementations of the $3$-RAR method (\cref{alg:r-RAR}). These experiments address two questions left open by the theory:
\begin{quote}
	\textit{Can third-order Riemannian adaptive regularization methods be implemented using modern subproblem solvers? How do those solvers behave when embedded within a Krylov-based framework?}
\end{quote}
The purpose of these experiments is not to optimize performance on this benchmark or to compete with highly optimized state-of-the-art RO packages. Rather, they are designed to assess the practical viability of the proposed $3$-RAR framework and to examine the behavior of modern third-order subproblem solvers within a Krylov-based implementation. Accordingly, our focus is on feasibility, convergence behavior, and relative solver performance under controlled experimental conditions.

Due to its importance in applications, we select a common RO experimental testbed, Brockett cost minimization \cite{Zhou26, Wang20, Xiaojing20}. Its practical applications include low-rank tensor completion and distributed learning \cite{Zhou26}; polynomial and combinatorial optimization and sparse PCA \cite{Wen13}; training of deep learning methods such as CNNs and RNNs \cite{LiLi20}; and matrix diagonalization and eigenvalue sorting \cite{Brockett91}. Formally, the Brockett cost minimization problem is
\begin{equation}\label{eq:brockett-cost}
	\min_{X \in \mathrm{St}(n,p)} \frac{1}{2}\operatorname{Tr}\!\left(X^\top A X N\right),
\end{equation}
where $A \in \mathbb{R}^{n \times n}$ and $N \in \mathbb{R}^{p \times p}$ are symmetric matrices, and the factor of $1/2$ is introduced for convenience.

In \cref{sec:numerics:setup} we describe the setup of our experiments, while in \cref{sec:numerics:results} we present and discuss their results.

\subsection{Setup and Implementation Details}\label{sec:numerics:setup}

In this subsection, we describe the experimental setup underlying all numerical results reported in this section. Four components determine a $3$-RAR implementation: the data generation scheme, algorithmic parameter initialization and update, the retraction, and the Euclidean subproblem solvers. The specification of the first three elements is consolidated in the list below:
\begin{enumerate}
	\item \textit{Experimental Grid \& Data Generation Scheme}:  We consider an experimental grid defined by five problem sizes and two values of the parameter $\theta$. The problem sizes and $\theta$ values considered are:
	\begin{align*}
		(n,p)&\in\{(10,5), (20,5), (50,5), (50,10), (100,5)\}\\
		\theta &\in \{0.25, 2.0\}
	\end{align*}
	For the first three problem sizes and for each value of $\theta$, all subproblem solvers are evaluated on $20$ independently generated instances. For the larger problem sizes $(50,10)$ and $(100,5)$, and for each value of $\theta$, only the best-performing solvers observed on the smaller problem sizes are evaluated on $20$ instances. The remaining solvers are tested on a single representative instance due to increased computational cost. For each instance, the matrix $A \in \bbR^{n \times n}$ is generated as a symmetric matrix with independent standard normal entries. The matrix $N$ is fixed throughout and set to $\Diag(1,\ldots,p)$.

\item \textit{Parameter Initialization \& Update}: 
For all experiments, we initialize \cref{alg:r-RAR} with the parameters 
$\eta_1 = 0.1$, $\eta_2 = 0.9$, $\gamma_1 = 0.1$, $\gamma_2 = \gamma_3 = 2$, 
$\regparam_{\min} = 10^{-10}$, and $\regparam_0 = 20$. 
These values are selected following the practice of \cite{Agarwal21}, with the exception of $\regparam_0$, which is automatically determined by Manopt in their implementation. The termination tolerance is set to $\epsilon_1 = 10^{-6}$ for all instances, which is consistent with the numerical experiments reported in the literature for the subproblem solvers under consideration. We update the regularization parameter $\regparam_i$ using the midpoint of the interval corresponding to the success of the current step in \eqref{eq:alg-sigma-update}.

\item \textit{Retraction}: We use the Stiefel manifold retraction \eqref{eq:Gawlik-Stiefel-retraction} of \cite{Gawlik18} with $m = 2$, which yields a third-order retraction by \cref{thm:p-th-order-Gawlik}.
\end{enumerate}

For the Euclidean subproblem solvers, we experiment with three modern frameworks, as well as a baseline implementation of gradient descent with line search. All solvers are invoked within the \nameref{alg:Krylov-Framework}. The individual solvers and their implementation details are summarized below.

\begin{itemize}

\item \textit{Diagonal Tensor Method (DTM)} \cite[Algorithm~4]{Cartis25_DTM}: This method solves the subproblem by iteratively minimizing a sequence of local models with diagonal cubic terms and quartic regularization. We use the following parameters: $\eta = 0.01, \; \eta_1 = 3, \; \gamma = 2, \; \gamma_2 = 0.5$. We follow the Diagonal Rule update and use ARC as the default solver for the secular equation, which is a safeguard method as seen in \cite[Algorithm~1]{Cartis25_DTM}.

\item  \textit{Quadratic Quartic Regularization, version 2 (QQR-v2)} \cite[Algorithm~3]{Cartis25_QQR}: This method\\ approximates the third-order tensor term by a linear combination of quadratic and quartic terms, yielding local models that are possibly non-convex but solvable to global optimality. We set the following parameters (in the notation of \cite{Cartis25_QQR}): $\rho_1 = 0.1, \; \rho_2 = 0.9, \; \eta_0 = 0.5, \; \eta_1 = 2, \; \gamma_2 = 1.5$. Our experiments employ the second variant.

\item \textit{Cubic Quartic Regularization, versions 1-3 (CQR-v1,-v2,-v3)} \cite[Algorithm~5]{Cartis25_CQR}: This method successively minimizes a sequence of local quadratic models that also incorporate simple cubic and quartic terms. We set the following parameters (in the notation of \cite{Cartis25_CQR}): $\eta_1=0.9, \; \eta_2=0.1, \; \gamma=2, \; \gamma_2=0.5$. Our implementation contains all three variants so that $\beta$ is initialized and updated according to the different rules presented.

\item \textit{Gradient Descent with Armijo Backtracking (Armijo-GD)}: For this algorithm, we set the Armijo constant $c_1 = 10^{-4}$, backtracking factor $\rho = 0.5$ and initial step size $\alpha = 1$. This is mainly meant to serve as a basis for comparison.
\end{itemize}
Beginning with the second iteration of the \nameref{alg:Krylov-Framework}, we warm start each subproblem solver by lifting the previously computed solution to the current Krylov subspace dimension. This strategy is used for all subproblem solvers except DTM, which is initialized at zero at every iteration in accordance with \cite[Algorithm~2]{Cartis25_DTM}.

All experiments use a MacBook Pro (2023) with an Apple M3 Max, 96 GB. The algorithms were implemented using Python 3.12.6, with the SciPy 1.16.2, NumPy 2.3.5 and Matplotlib 3.10.7 packages.

\subsection{Results and Discussion}\label{sec:numerics:results}

In this subsection, we present and discuss the results of our experiments. We find that DTM and QQR-v2 vie for the best performance out of all solvers tested in terms of runtime. Such behavior from both solvers may be indicative of structural properties of the underlying third-order tensor and invites further investigation. CQR appears to occasionally result in fewer $3$-RAR iterations than other solvers, with such behavior seemingly more common as problem size grows.

We present two figures and three tables that summarize our results. Figures \ref{fig:norm-vs-iterates} and \ref{fig:norm-vs-time} plot the gradient norm against total $3$-RAR iterations and time elapsed, respectively, for representative instances of our setups with $(n,p)$=$(10,5)$, $(20,5)$, and $(50,5)$, and $\theta=0.25$ and $2$. Table \ref{table:Krylov-full-max} shows the results in terms of average iteration count and total time until convergence across all runs of our experiments. Table \ref{table:50-10-Krylov-full-max} shows the averaged results applied to St$(100,5)$ and St$(50,10)$ employing the DTM and the QQR-v2 subproblem solvers, which presented the best performances out of all methods. It also contains the direct results for a single instance of the other solvers, as running all 20 tests for each $\theta$ with each of those subproblem solvers proved to be prohibitively expensive.

\begin{figure}[ht!]
	\centering
	\includegraphics[width=\textwidth]{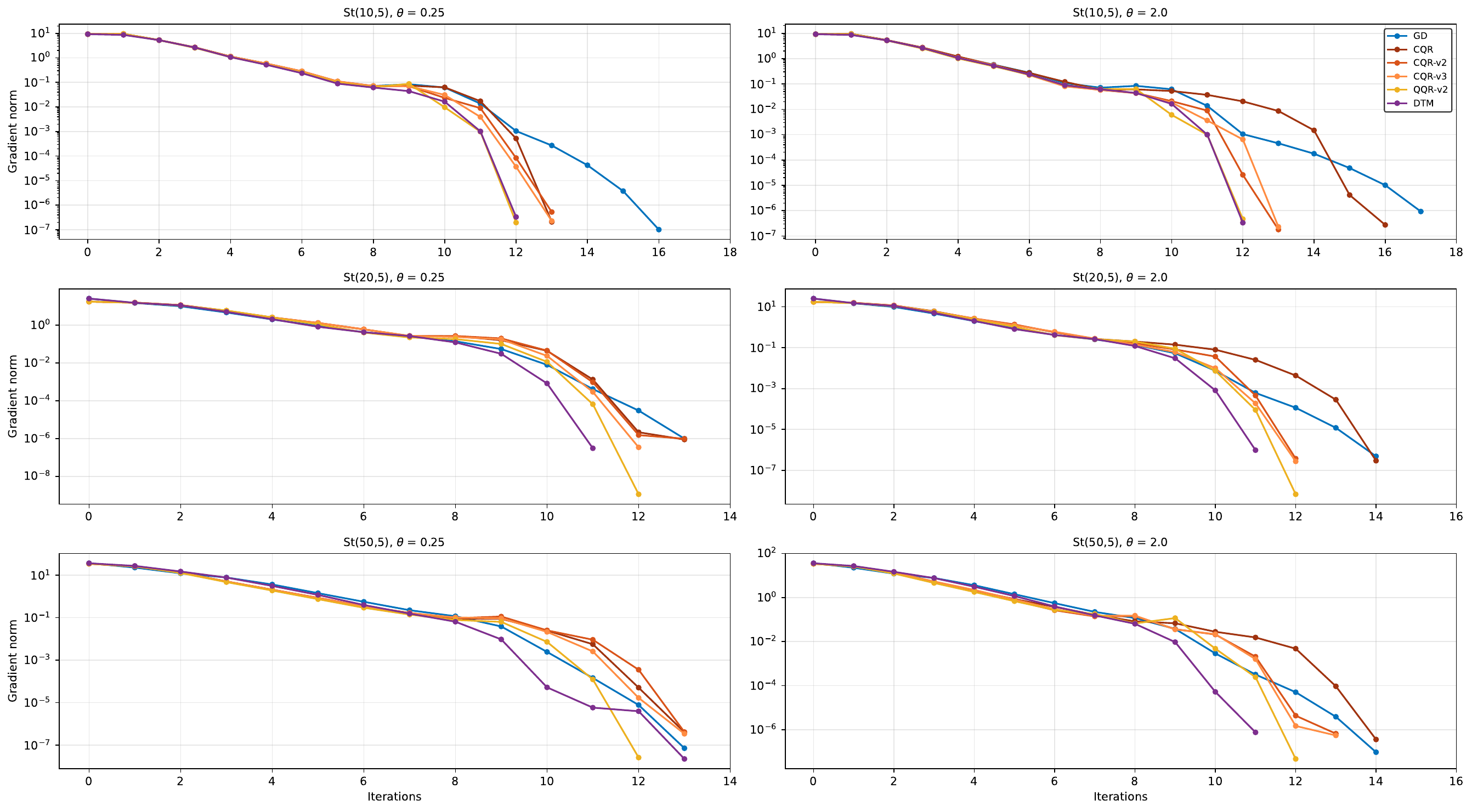}
	\caption{Gradient norm versus $3$-RAR iteration count for representative instances on $(10,5)$, $(20,5)$, and $(50,5)$ setups.}
	\label{fig:norm-vs-iterates}
\end{figure}

\begin{figure}[ht!]
	\centering
	\includegraphics[width=\textwidth]{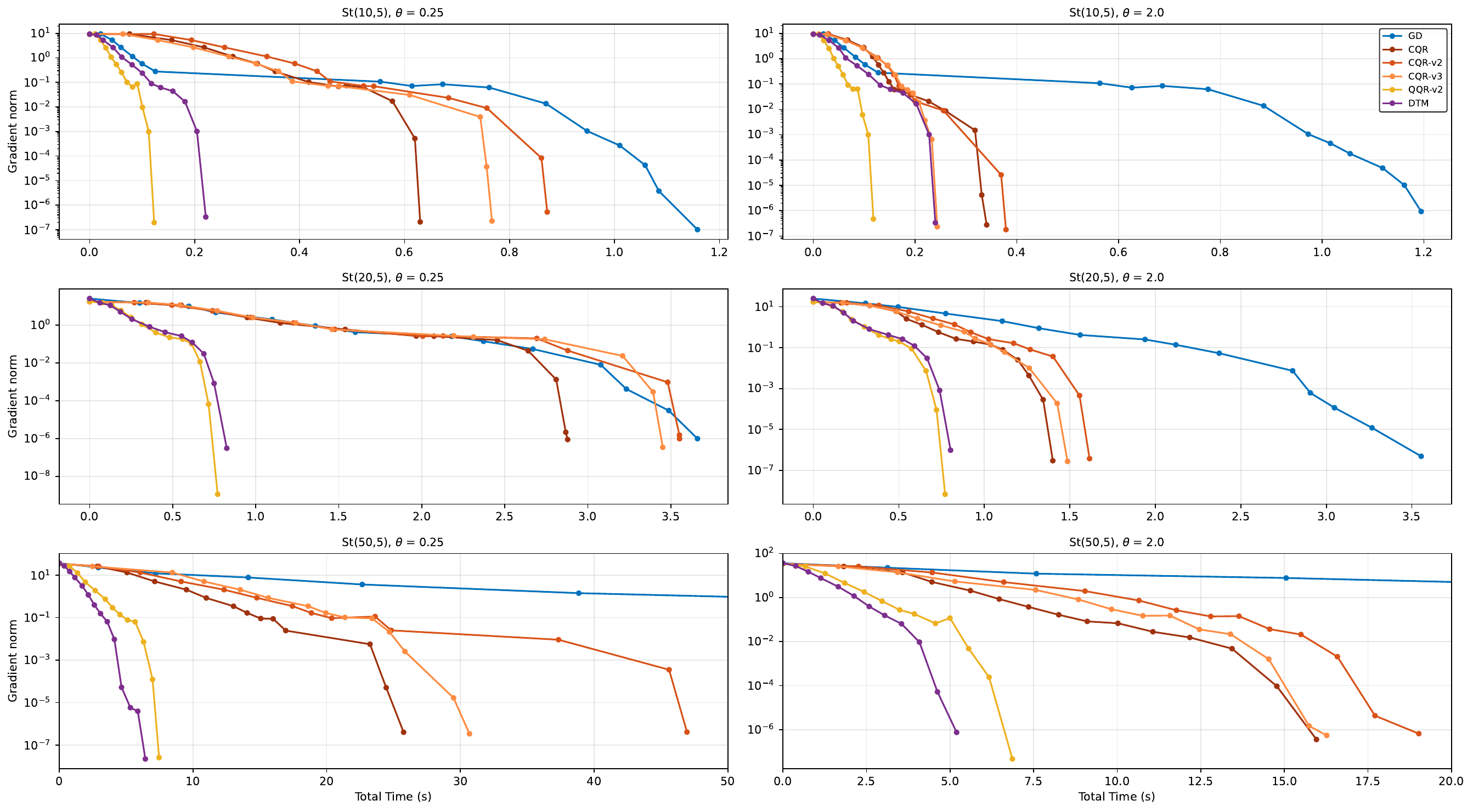}
	\caption{
		Gradient norm versus total CPU time for representative instances on $(10,5)$, $(20,5)$, and $(50,5)$ setups.}
	\label{fig:norm-vs-time}
\end{figure}

\begin{table}[h!]
	\small
	\centering
	\renewcommand{\arraystretch}{1.25}
	\setlength{\tabcolsep}{4pt}
	\resizebox{\textwidth}{!}{%
		\begin{tabular}{llrlrlrlrlrlrl}
    \toprule
    \multicolumn{2}{c}{} &
    \multicolumn{2}{c}{\textbf{DTM}} &
    \multicolumn{2}{c}{\textbf{QQR-v2}} &
    \multicolumn{2}{c}{\textbf{CQR-v1}} &
    \multicolumn{2}{c}{\textbf{CQR-v2}} &
    \multicolumn{2}{c}{\textbf{CQR-v3}} &
    \multicolumn{2}{c}{\textbf{ArmijoGD}} \\
    \cmidrule(lr){3-4}\cmidrule(lr){5-6}\cmidrule(lr){7-8}\cmidrule(lr){9-10}\cmidrule(lr){11-12}\cmidrule(lr){13-14}
    \textbf{(n,p)} & $\boldsymbol{\theta}$ &
    \textbf{Iter} & \textbf{Time} &
    \textbf{Iter} & \textbf{Time} &
    \textbf{Iter} & \textbf{Time} &
    \textbf{Iter} & \textbf{Time} &
    \textbf{Iter} & \textbf{Time} &
    \textbf{Iter} & \textbf{Time} \\
    \midrule\midrule
    \multirow{2}{*}{(10, 5)} &
        0.25 & 11.7 & 0.227 & \textbf{10.6} & \textbf{0.107} & 11.4 & 0.745 & 11.1 & 0.790 & 10.9 & 0.731 & 14.4 & 1.121 \\
    &   2    & 12.1 & 0.239 & \textbf{10.4} & \textbf{0.097} & 12.8 & 0.265 & 11.2 & 0.337 & 11.2 & 0.318 & 16.2 & 1.178 \\
    \midrule
    \multirow{2}{*}{(20, 5)} &
        0.25 & 12.0 & 0.907 & \textbf{11.9} & \textbf{0.723} & 12.2 & 2.580 & 12.0 & 3.520 & 12.0 & 3.040 & 13.5 & 4.615 \\
    &   2    & 12.1 & 0.897 & \textbf{11.7} & \textbf{0.684} & 13.8 & 1.306 & 12.1 & 1.443 & 12.1 & 1.476 & 14.6 & 4.599 \\
    \midrule
    \multirow{2}{*}{(50, 5)} &
        0.25 & 12.3 & \textbf{5.758} & \textbf{12.0} & 7.543 & 12.6 & 30.451 & 12.4 & 47.764 & 12.3 & 35.202 & 14.3 & 142.144 \\
    &   2    & 12.6 & \textbf{5.893} & \textbf{12.0} & 6.600 & 14.3 & 15.044 & 12.7 & 19.396 & 12.8 & 17.272 & 15.8 & 136.756 \\
    \bottomrule
\end{tabular}
	}
	\caption{Mean $3$-RAR iteration count and convergence time (sec.) across all $(10,5)$, $(20,5)$, and $(50,5)$ instances.}
	\label{table:Krylov-full-max}
\end{table}

\begin{table}[h!]
	\small
	\centering
	\renewcommand{\arraystretch}{1.25}
	\setlength{\tabcolsep}{4pt}
	\resizebox{\textwidth}{!}{%
		\begin{tabular}{llrlrlrlrlrlrl}
    \toprule
    \multicolumn{2}{c}{} &
    \multicolumn{2}{c}{\textbf{DTM}} &
    \multicolumn{2}{c}{\textbf{QQR-v2}} &
    \multicolumn{2}{c}{\textbf{CQR-v1*}} &
    \multicolumn{2}{c}{\textbf{CQR-v2*}} &
    \multicolumn{2}{c}{\textbf{CQR-v3*}} &
    \multicolumn{2}{c}{\textbf{ArmijoGD*}} \\
    \cmidrule(lr){3-4}\cmidrule(lr){5-6}\cmidrule(lr){7-8}\cmidrule(lr){9-10}\cmidrule(lr){11-12}\cmidrule(lr){13-14}
    \textbf{(n,p)} & $\boldsymbol{\theta}$ &
    \textbf{Iter} & \textbf{Time} &
    \textbf{Iter} & \textbf{Time} &
    \textbf{Iter} & \textbf{Time} &
    \textbf{Iter} & \textbf{Time} &
    \textbf{Iter} & \textbf{Time} &
    \textbf{Iter} & \textbf{Time} \\
    \midrule\midrule
    \multirow{2}{*}{(100, 5)} &
        0.25 & 15.3 & \textbf{34.652} & 15.8 & 76.086 & \textbf{13.0} & 242.081 & \textbf{13.0} & 325.610 & \textbf{13.0} & 290.314 & 18.0 & 1142.009 \\
    &   2    & 15.8 & 35.769 & 15.0 & 60.815 & 17.0 & 147.712 & 15.0 & 192.125 & 14.0 & 141.006 & 20.0 & 1171.022 \\
    \midrule
    \multirow{2}{*}{(50, 10)} &
        0.25 & 17.9 & \textbf{50.741} & 17.1 & 79.875 & 17.0 & 377.565 & \textbf{16.0} & 883.647 & \textbf{16.0} & 617.605 & 20.0 & 7034.250 \\
    &   2    & 22.6 & 68.856 & 16.5 & 63.724 & 23.0 & 191.535 & \textbf{16.0} & 168.587 & \textbf{16.0} & 184.095 & 23.0 & 6653.714 \\
    \bottomrule
\end{tabular}
	}
	\caption{
		Mean $3$-RAR iteration count and convergence time (sec.) across all $(100,5)$ and $(50,10)$ instances. * means only one test run was executed.}
	\label{table:50-10-Krylov-full-max}
\end{table}

With these numerical results in hand, and reiterating our goals of validating the proposed algorithms and comparing recently developed solvers, a few considerations are in order. From Table \ref{table:Krylov-full-max} one can notice that, for the smaller problem sizes of St$(10,5)$ and St$(20,5)$, QQR-v2 presents the best runtimes of all solvers. For St$(50,5)$, DTM seems to exhibit a better performance; this also holds for most of the larger test instances presented in Table \ref{table:50-10-Krylov-full-max}. With these observations, DTM and QQR-v2 appear to be the solvers best suited for the present use case. All three versions of the CQR algorithm seem to exhibit runtimes that become faster relatively to the Armijo line-search Gradient Descent as problem size grows. CQR also seems to result in fewer $3$-RAR iterations, a behavior seemingly more common with increasing problem size as evidenced by Table \ref{table:50-10-Krylov-full-max}. For all of these observations, we reiterate that only common Python libraries were used, and differences in the implementations of each subproblem solver may affect their relative performances. 

The experiments in \cite{Cartis25_DTM} show that, under special conditions for the third-order tensor, DTM appears to enjoy superior performance compared to other methods like ARC or CQR. The consistently better runtimes demonstrated by DTM in our experiments might be indicative of some special structure in our third-order tensor. The presence of these properties in the case of our algorithm, as well as the ones mentioned in \cref{sec:Krylov-iteration}, can be objects of future study. Furthermore, we expect future work on aspects such as initial parameter selection and update rules on the regularization parameter to improve runtimes for our methods. As an example of this, \cite{Agarwal21} cites \cite{Gould12} as a piece of work closely related to this matter.

\section{Conclusion and Future Directions}\label{sec:conclusion}

This work proves that unconstrained smooth nonconvex RO, for all orders, is no harder from an oracle complexity perspective than its Euclidean counterpart. In pursuit of this goal, we devised a clean and reusable framework for studying the regularity of retraction-based pullback functions. On the implementation side, we resolved the critical obstacles that have thus far impeded the practical realization of higher-order Riemannian methods. These results are not meant to be the final word on higher-order RO, but rather an invitation to the community for further exploration. In particular, our results open several promising avenues for future research:
\begin{itemize}
	\item \textit{Polynomial-time higher-order RO methods}: Recent work in \cite{Ahmadi24} established polynomial-time guarantees for higher-order methods in Euclidean space via structured regularization and convex relaxations. Developing analogous polynomial-time higher-order Riemannian methods that integrate such ideas with the geometric framework introduced here is an important open direction.
	
	\item \textit{Proximal and composite methods}: Extending higher-order Riemannian methods and the associated smoothness analysis to proximal and composite settings would substantially broaden their applicability.
	
	\item \textit{Higher-order methods for geodesically convex RO}: Although this paper resolves the nonconvex complexity landscape, the role of higher-order methods in geodesically convex RO remains far less understood. 
\end{itemize}

We hope that the tools and perspectives developed here help position higher-order RO as a mature area on equal theoretical footing with its Euclidean analogue, and stimulate further advances at the interface of geometry, optimization, and computation.

\subsection*{Acknowledgements}

The authors are very grateful to Kate Wenqi Zhu and Coralia Cartis for their fruitful discussions about their subproblem solvers.

\appendix

\section{Appendix to \cref{sec:retraction-properties}}\label{app:retraction-properties}

\subsection{Proof of the \nameref{lemma:ordered-partition-properties}  (\cref{lemma:ordered-partition-properties})}\label{app:ordered-partition-properties}
Suppose that $\sfB=(\sfB_1,\ldots,\sfB_k)\in\ordpart(p+1,k)$. We must show that $\sfB=\insertind_{p+1}(\tilde{\sfB},i)$ for some $1\leq i\leq k$ or  $\tilde{\sfB}\in \ordpart(p,k)$ for some $\tilde{\sfB}\in\ordpart(p,k-1)$. There are two cases to consider: $\sfB_k=\{p+1\}$ and $\sfB_k\neq\{p+1\}$. In the first, we automatically have that $\tilde{B}:=(\sfB_1,\ldots,\sfB_{k-1})\in\ordpart(p,k-1)$, so $\sfB=\appendind_{p+1}\left(\tilde{\sfB}\right)$. In the second case, it must hold that $\hat{\sfB}:=(\sfB_1,\ldots ,\sfB_i\setminus\{p+1\},\ldots, \sfB_k)\in\ordpart(p,k)$ for some $1\leq i\leq k$, so $\sfB=\insertind_{p+1}\left(\hat{\sfB}\right)$.\qed

\subsection{Proof of \nameref{cor:pullback-function-derivative-formula-non-total} (\cref{cor:pullback-function-derivative-formula-non-total})} \label{app:pullback-function-derivative-formula-non-total}

We will need a combinatorial identity for the proof.

\begin{lemma}\label{lemma:multinomial-partition-equality}
	Let $a_1,\ldots,a_{p+1}\geq 0$. It holds that
	\begin{equation}\label{eq:multinomial-partition-equality}
		\sum_{B\in\ordpart(p+1,k)} \prod_{i=1}^k a_{|\sfB_i|}=\frac{1}{k!}\sum_{\ell\vDash_k p+1}\binom{p+1}{\boldL} \prod_{i=1}^k a_{\ell_i}.
	\end{equation}
\end{lemma}

\begin{proof}
	Let $\cC_{p+1,k}$ be the set of all ordered partitions of $[p+1]$ into $k$ non-empty subsets, i.e. the set of $k$-tuples of non-empty subsets $\sfC=(\sfC_1,\ldots,\sfC_k)$. Observe that each $B\in\ordpart(p+1,k)$ corresponds to $k!$ elements $\cC(p+1,k)$, namely the set of $(B_{\phi(1)},\ldots,B_{\phi(k)})$ where $\phi\in S_k$. This implies the set equalities
	\[
	\bigsqcup_{\boldL\vDash_k p+1}\{C\in \cC(p+1,k):|\sfC|=\boldL\}=\cC(p+1,k)=\bigsqcup_{B\in\ordpart(p+1,k)}\{(B_{\phi(1)},\ldots,B_{\phi(k)}):\phi\in S_k\}.
	\]
	Thus, because $|S_k|=k!$ and $\{\sfC\in \cC(p+1,k):|\sfC|=\boldL\}=\binom{p+1}{\boldL}$, it holds that
	\begin{multline*}
		\sum_{\boldL\vDash_k p+1}\binom{p+1}{\boldL}\prod_{i=1}^k a_{\ell_i}=\sum_{\boldL\vDash_k p+1}\sum_{\substack{\sfC\in \cC(p+1,k)\\ |\sfC|=\boldL}}\prod_{i=1}^k a_{\ell_i}=\sum_{\sfC\in\cC(p+1,k)}\prod_{i=1}^k a_{|C_i|}\\
		=\sum_{B\in\ordpart(p+1,k)}\sum_{\phi\in S_k}\prod_{i=1}^k a_{|B_{\phi(i)}|}=k!\cdot \sum_{B\in\ordpart(p+1,k)}\prod_{i=1}^k a_{|B_i|}.
	\end{multline*}
\end{proof}

Now, we can prove the main bound.

\begin{proof}[Proof of \nameref{cor:pullback-function-derivative-formula-non-total} (\cref{cor:pullback-function-derivative-formula-non-total})]
	Fix $v,\tilde{v}\in\cV$ with $\pi(v)=\pi(\tilde{v})=:x$, and let $x_+=R(v)$. By the \nameref{thm:covariant-faa-di-bruno} \eqref{eq:faa-di-bruno-covariant}, with $X_1,\ldots,X_{p+1}=\gamma'$, it holds that
	\[
	D^{p+1}\hat{f}_x(v)\left[\tilde{v}^{\otimes (p+1)}\right]=\left.\frac{d^{p+1}}{dt^{p+1}}\hat{f}_x(v+t\tilde{v})\right|_{t=0}=\sum_{k=1}^{p+1}\left(\nabla^k f\right)_{x_+}\circ\sum_{B\in\ordpart(p+1,k)}\left(\prod_{i=1}^k D_t^{|B_i|-1}\gamma'(0)\right).
	\]
	Thus, the combinatorial equality \eqref{eq:multinomial-partition-equality} of the last lemma, along with the triangle inequality and norm-definining equalities, implies
	\begin{multline*}
		\|D^{p+1}\hat{f}_x(v)\left[\tilde{v}^{\otimes (p+1)}\right]\|\leq \sum_{k=1}^{p+1}\left\|\left(\nabla^k f\right)_{x_+}\right\|_{op}\cdot\sum_{\sfB\in\ordpart(p+1,k)}\prod_{i=1}^k \left\|D_t^{|B_i|-1}\gamma'(0)\right\|\\
		=\sum_{k=1}^{p+1}\frac{\|(\nabla^k f)_{x_+}\|_{op}}{k!}\cdot \sum_{\ell\vDash_k p+1}\binom{p+1}{\boldL}\prod_{i=1}^k \left\|D_t^{\ell_i-1}\gamma'(0)\right\|\\
		=(p+1)!\cdot\sum_{k=1}^{p+1}\frac{\|(\nabla^k f)_{x_+}\|_{op}}{k!}\cdot \sum_{\ell\vDash_k p+1}\prod_{i=1}^k \frac{ \left\|D_t^{\ell_i-1}\gamma'(0)\right\|}{\ell_i!}\\
		\leq (p+1)!\cdot\max_{1\leq k\leq p+1}\frac{\|(\nabla^k f)_{x_+}\|_{op}}{k!}\cdot \sum_{k=1}^{p+1}\sum_{\ell\vDash_k p+1}\prod_{i=1}^k \frac{ \left\|D_t^{\ell_i-1}\gamma'(0)\right\|}{\ell_i!}.
	\end{multline*}
\end{proof}

\subsection{Proof of \nameref{cor:pullback-function-derivative-formula-total} (\cref{cor:pullback-function-derivative-formula-total})} \label{app:pullback-function-derivative-formula-total}

	Fix $v\in\cV$ and $\tilde{v}\in\cT_{\pi(v)}\cM$ with $\|\tilde{v}\|_{\pi(v)}=1$. Using the second item from the \nameref{lemma:TM-calculus} (\cref{lemma:TM-calculus}), we have that
	\[
	D_t^{\ell_i-1} \gamma'(0)=(\nabla^R)^{\ell_i} R_v\left[\vl_v(\tilde{v})^{\otimes (p_i-1)}\right]=\Sym\left[(\nabla^R)^{\ell_i} R_v\right]\left(\vl_v(\tilde{v})^{\otimes \ell_i}\right),
	\]
	when $\gamma(t)=R(v+t\tilde{v})$. Thus, we can process the \nameref{cor:pullback-function-derivative-formula-non-total} \eqref{eq:pullback-derivative-non-total} into
	\begin{equation}
		\left\|D^{p+1}\hat{f}_{\pi(v)}(\tilde{v})\right\|\leq (p+1)!\cdot\|f\|_{p+1,R(\cV)}\cdot \sum_{k=1}^{p+1} \sum_{\ell\vDash_k p+1}\prod_{i=1}^k \frac{ \left\|\Sym\left[(\nabla^R)^{\ell_i} R_v\right]\left(\vl_v(\tilde{v})^{\otimes \ell_i}\right)\right\|}{\ell_i!}.
	\end{equation}
	Now, using the fact that there are $\binom{p-1}{k-1}$ of $\boldL$ with $\boldL\vDash_k p$ and various operator bounds, we arrive at the final bound via
	\begin{multline*}
		\left\|D^{p+1}\hat{f}_{\pi(v)}(\tilde{v})\right\|\leq (p+1)!\cdot\|f\|_{p+1,R(\cV)}\cdot \sum_{k=1}^{p+1} \sum_{\ell\vDash_k p+1}\prod_{i=1}^k \frac{ \left\|\Sym\left[(\nabla^R)^{\ell_i} R_v\right]\left(\vl_v(\tilde{v})^{\otimes \ell_i}\right)\right\|}{\ell_i!}\\
		\leq (p+1)!\cdot\|f\|_{p+1,R(\cV)}\cdot \left[\sum_{k=1}^{p+1} \sum_{\ell\vDash_k p+1}\|R\|_{p,\cV}^k\right]\cdot\left\|\tilde{v}\right\|^{p+1}\\
		\leq p!\cdot\|f\|_{p,R(\cV)}\cdot \left[\sum_{k=1}^{p+1} \binom{p}{k-1}\|R\|_{p,\cV}^k\right]\leq (p+1)!\cdot\|f\|_{p+1,R(\cV)}\cdot \|R\|_{p+1,\cV}(1+\|R\|_{p+1,\cV})^p.\qed
	\end{multline*}

\subsection{Proof of \nameref{lemma:Taylor-vector-field} (\cref{lemma:Taylor-vector-field})} \label{app:Taylor-vector-field}
	
		The inequality \eqref{eq:Taylor-vector-field-bound} follows immediately from  \eqref{eq:Taylor-vector-field} by applying the triangle inequality for integrals along with isometry of the parallel transport. Thus, we only prove \eqref{eq:Taylor-vector-field}.
	
	Let $E_1,\ldots,E_n$ be a parallel orthonormal frame along $\gamma$. Then there exist smooth functions $V_1,\ldots,V_n:[0,1]\to\bbR$ such that $V(s)=\sum_{j=1}^n V_j(s)E_i(s)$. Applying the standard one-dimensional Taylor's theorem with remainder to each $V_j$, and the fact that $\partrans_{t \to 0}E_j(s)=E(0)$, we compute
	\begin{align*}
		\partrans_{t \to 0}E(t)&=\sum_{j=1}^n V_j(s)	\partrans_{t \to 0}E_j(s)\\
		&\begin{aligned}
			&=\sum_{j=1}^n \left[V_j(0)\right]E_j(0)+\sum_{i=1}^k\sum_{j=1}^n \left[\frac{V_j^{(i)}(0)}{i!}t^i\right]E_j(0)\\
			&\qquad\qquad+\frac{1}{k!}\int_0^t (t-s)^k\;\sum_{j=1}^n\left[V_j^{(k+1)}(s)E_j(0)\right]\;ds\\
		\end{aligned}\\
		&=V(0)+\sum_{i=1}^k\frac{1}{i!}\left[\left(\frac{D}{dt}\right)^i V(0)\right]t^i+\frac{1}{k!}\int_0^t (t-s)^k\;\partrans_{s \to 0}\left[\left(\frac{D}{dt}\right)^{k+1} V(s)\right]\;ds.
	\end{align*}
	This establishes the main equality \eqref{eq:Taylor-vector-field}. \qed

\subsection{Proof of \nameref{lemma:iterate-compactness} (\cref{lemma:iterate-compactness})}\label{app:iterate-compactness}

 This proof depends upon a generalization of eigenvalues for symmetric operators to symmetric tensors. Following \cite{Qi17}, given a symmetric $k$-tensor $T$ on an inner product space $V$, we call
\[
\lambda_{\min}(T):=\min_{v\in V\atop \|v\|=1}\left|T\left(v^{\otimes k}\right)\right|
\]
the smallest \emph{Z-eigenvalue} of $T$. We will also call up Fujiwara's bound for a polynomial $P(z)=\sum_{i=0}^n a_ix_i$, which says that $P$'s roots are bounded above in modulus by 
\[
2\cdot\max_{1\leq k\leq n-1}\left\{\left|\frac{a_{n-k}}{a_n}\right|^{\frac{1}{k}},\left|\frac{a_{0}}{2a_n}\right|^{\frac{1}{n}}\right\}.
\]

\begin{proof}
	The second item is an immediate consequence of the first and continuity of $R$. Let $\cX=\{x_i\}_{i=1}^\infty$. Thus, we need only prove the first item. By $p$-RAR's model decrease condition \eqref{eq:alg-model-decrease-condition}, which states that $m_p(x_i,\regparam_i)[v_i]\leq m_p(x_i,\regparam_i)[0]$, it holds that
	\begin{multline*}
		\sum_{j=0}^p \frac{1}{j!}D^j \hat{f}_{x_i}(0)\left[v_i^{\otimes j}\right]+\frac{\regparam_i}{p+1}\|v_i\|_{x_i}^{p+1}=T_p \hat{f}_{x_i}(0)[v_i]+\frac{\regparam_i}{p+1}\|v_i\|_{x_i}^{p+1}=m_p(x_i,\regparam_i)[v_i]\\
		\leq m_p(x_i,\regparam_i)[0]=f(x_i).
	\end{multline*}
	Rearranging, and using the fact that $\regparam_i\geq\regparam_{\min}$ for all $i\geq 1$  we see that
	\[
	\frac{\regparam_{\min}}{p+1}\|v_i\|_{x_i}^{p+1}\leq -\sum_{j=1}^p \frac{1}{j!}D^j \hat{f}_{x_i}(0)\left[v_i^{\otimes j}\right].
	\]
	We will now decompose the right-hand side so we can bound $v_i$ in terms of the Z-eigenvalues of $\hat{f}_{x_i}$'s derivatives. We have that
	\begin{align*}
		-\sum_{j=1}^p \frac{1}{j!}D^j \hat{f}_{x_i}(0)\left[v_i^{\otimes j}\right]&=D\hat{f}_{x_i}(0)[-v_i]-\sum_{j=2}^p \frac{1}{j!}D^j \hat{f}_{x_i}(0)\left[v_i^{\otimes j}\right]\\
		&\leq \|\grad f(x_i)\|_{x_i}\|v_i\|_{x_i}+\sum_{j=2}^p \frac{1}{j!}\max\{0,-\lambda_{\min}(\nabla^j f(x_i))\}\left\|v_i\right\|_{x_i}^j.
	\end{align*}
	Thus, using the fact that $R$ is a $p$-th order retraction,
	\[
	\frac{\regparam_{\min}}{p+1}\|v_i\|_{x_i}^p-\sum_{j=2}^p \frac{1}{j!}\max\{0,-\lambda_{\min}(\nabla^j f(x_i))\}\left\|v_i\right\|_{x_i}^{j-1}-\|\grad f(x_i)\|_{x_i}\leq 0
	\]
	which implies that $\|v_i\|_{x_i}$ is bounded above by the largest root of the polynomial
	\[
	t\mapsto \frac{\regparam_{\min}}{p+1}t^p-\sum_{j=2}^p \frac{1}{j!}\max\{0,-\lambda_{\min}(\nabla^j f(x_i))\}t^{j-1}-\|\grad f(x_i)\|_{x_i}.
	\]
	By the Fujiwara bound, $\|v_i\|_{x_i}$ satisfies the bound
	\[
	\|v_i\|_{x_i}\leq 2\cdot\max_{1\leq j\leq p-1}\left\{\left(\frac{p+1}{j!}\cdot\frac{\max\{0,-\lambda_{\min}(\nabla^{p-j+1} f(x_i))\}}{\regparam_{\min}}\right)^{\frac{1}{p-j}},\left(\frac{p+1}{2}\cdot\frac{\|\grad f(x_i)\|_{x_i}}{\regparam_{\min}}\right)^{\frac{1}{p}}\right\}.
	\]
	The right-hand side of this latest inequality defines a continuous function of $x$ on $\cM$. Hence, it is bounded above by some $r>0$ on the closure of $\cX$, which is compact by hypothesis.  Thus, 
	\[
	\cV\subseteq \{v\in \pi^{-1}(\bar{\cX}): \|v\|_{\pi(x)}\leq r\},
	\]
	which is compact, and establishes the desired conclusion.
\end{proof}

\section{Appendix to \cref{sec:retractions-practical}}\label{app:retractions-practical}

\subsection{Proof of \cref{lemma:well-defined-retraction}}\label{app:well-defined-retraction}

The original article \cite{Gawlik18} already shows these maps are well-defined on their domain, so we need only show their domains are the entire tangent bundle. This means we must show that $\Theta_n(H)$ is full-rank for all $H\in\Skew(n)$. Every real skew-symmetric matrix is orthogonally equivalent to a block diagonal matrix whose blocks are either zero or $2\times 2$ blocks of the form $H = aJ$ for some $a \in \bbR$, where $J=\begin{bmatrix}
	0 & -1\\
	1 & 0
\end{bmatrix}$. Consequently, it suffices to consider the $n=2$ case.

For $n=2$, we can write $H$ as $H=aJ$ with $J$ as above. For all $k\geq 0$, it holds that $J^{4k}=I$, $J^{4k+1}=J$, $J^{4k+2}=-I$, and $J^{4k+3}=-J$, so 
	\begin{equation}\label{eq:H-powers}
		H^k = 
		\begin{cases}
			(-1)^{k/2} a^k I & \text{if } k \text{ is even} \\
			(-1)^{(k-1)/2} a^k J & \text{if } k \text{ is odd}
		\end{cases}.
	\end{equation}
	Accordingly, we will split $\Theta_n$ into its even and odd parts, i.e. the sum of the terms in $\Theta_n$ which have purely even or odd power, which we respectively write as $E_n$ and $O_n$. In other words, we have that
	\[
	E_n(z)=\sum_{k=0}^{\left \lfloor n/2\right\rfloor} \binom{n}{2k} \frac{(2n - 2k)!}{(2n)!} (2z)^{2k},\qquad O_n(z)=\sum_{k=0}^{\left \lfloor n/2\right\rfloor+1} \binom{n}{2k} \frac{(2n - (2k+1))!}{(2n)!} (2z)^{2k+1}.
	\]
	Plugging $H$ into each of these polynomials and using the expression \eqref{eq:H-powers} for the powers of $H$, simple matrix arithmetic shows $E_n(H)=E_n(ia)I$ and $O_n(H)=\frac{O_n(ia)}{i}J$. Thus, we have that
	\[
	\Theta_n(H)=E_n(H)+O_n(H)=\begin{bmatrix}E_n(ia) & -\frac{O_n(ia)}{i}\\ \frac{O_n(ia)}{i} & E_n(ia)\end{bmatrix}.
	\]
	This matrix is full-rank if and only if its determinant is non-zero. The determinant is readily seen to be $\det\left(	\Theta_n(H)\right)=E_n(ia)^2-O_n(ia)^2$. Noticing that $\Theta_n(ia)=E_n(ia)+i\left(O_n(ia)/i\right)$, we find that $E_n(ia)$ and $\left(\frac{O_n(ia)}{i}\right)$ are respectively the real and imaginary parts of $\Theta_n(ia)$. Thus, $\|\Theta_n(ia)\|^2=E_n(ia)^2-O_n(ia)^2$, and we have that the determinant is zero if and only if $ia$ is a root of $\Theta_n$. However, by \cite[Ch. 10, Theorem 5]{Grosswald06}, any root of $\Theta_n$ must have a strictly negative real part, which establishes the claim in the $n=2$ case, because $ia$ is purely imaginary.\qed

\bibliographystyle{abbrvnat}
\bibliography{fullBib}

\end{document}